\documentclass[reqno]{amsart}
\usepackage{float}
\usepackage{amsmath}
\usepackage{graphicx} 
\usepackage{latexsym}
\usepackage{amsfonts}
\usepackage{amssymb}
\usepackage{color} 
\theoremstyle{plain}
\newtheorem{theorem}{Theorem}
\newtheorem{corollary}[theorem]{Corollary}
\newtheorem{lemma}[theorem]{Lemma}
\newtheorem{proposition}[theorem]{Proposition}
\theoremstyle{definition}

\newtheorem{definition}[theorem]{Definition}

\newtheorem{remark}[theorem]{Remark}
\newtheorem*{remark*}{Remark}

\newcommand{\pr}{\mathbf P}
\newcommand{\e}{\mathbf E}
\newcommand{\E}{\mathbf E}

\begin{document}
\title[Persistence of autoregressive sequences with logarithmic tails]
    {Persistence of autoregressive sequences with logarithmic tails}
    \thanks{
        D. Denisov was supported by a Leverhulme Trust Research Project Grant  RPG-2021-105. 
        V. Wachtel was partially supported by DFG.
    }
\author[Denisov]{Denis Denisov}
\address{Department of Mathematics, University of Manchester, Oxford Road, Manchester M13 9PL, UK}
\email{denis.denisov@manchester.ac.uk}
    
\author[Hinrichs]{G\"unter Hinrichs} 
\address{Institut f\"ur Mathematik, Universit\"at Augsburg, 86135 Augsburg, Germany}
\email{Guenter.Hinrichs@math.uni-augsburg.de}

\author[Kolb]{Martin Kolb}
\address{Institut f\"ur Mathematik, Universit\"at Paderborn,
33098 Paderborn, Germany}
\email{kolb@math.uni-paderborn.de}

\author[Wachtel]{Vitali Wachtel}
\address{Faculty of Mathematics, Bielefeld University, Germany}
\email{wachtel@math.uni-bielefeld.de}\begin{abstract}
We consider autoregressive sequences $X_n=aX_{n-1}+\xi_n$ and
$M_n=\max\{aM_{n-1},\xi_n\}$ with a constant $a\in(0,1)$ and with positive, independent and identically distributed innovations $\{\xi_k\}$. It is known that if $\pr(\xi_1>x)\sim\frac{d}{\log x}$ with some $d\in(0,-\log a)$ then the chains $\{X_n\}$ and $\{M_n\}$ are null recurrent. We investigate the tail behaviour of recurrence times in this case of logarithmically decaying tails. More precisely, we show that the tails of recurrence times are regularly varying of index $-1-d/\log a$.
We also prove limit theorems for $\{X_n\}$ and $\{M_n\}$ conditioned to stay over a fixed level $x_0$.\\
Furthermore, we study tail asymptotics for recurrence times 
of $\{X_n\}$ and $\{M_n\}$ in the case when these chains are positive recurrent and the tail of $\log\xi_1$ is subexponential.
\end{abstract}
    
    
    \maketitle
    {\scriptsize
    }
\section{Introduction.}
Let $\{\xi_n\}_{n\ge 1}$ be 
a sequence of  independent and  identically distributed random variables. 
Let $a\in(0,1)$ be a constant. 
The corresponding AR($1$)-sequence $X=\{X_n\}_{n\ge 0}$ is defined by 
$$
X_{n}:=aX_{n-1}+\xi_{n},\quad n\ge 1,
$$
where the starting position $X_0$ can be either random or deterministic. 

Besides the Markov chain $X$ we shall consider the so-called 
{\it maximal autoregressive sequence} 
$M=\{M_n\}_{n\ge 0}$, where 
$$
M_n=\max\{aM_{n-1},\xi_n\},\quad n\ge1.
$$
The Markov chains $X$ and $M$ have rather similar properties.
If, for example, the innovations are non-negative then these two chains are recurrent, positive recurrent or transient at the same time. More precisely, according to Theorem 3.1 in Zerner~\cite{Zerner18}, the chains $\{X_n\}$ and $\{M_n\}$ are recurrent if and only if
\begin{equation}
\label{eq:recurrence}
\sum_{n=0}^\infty\prod_{m=0}^n\pr(|\xi_1|\le ta^{-m})=\infty
\end{equation}
for every $t$ satisfying $\pr(|\xi_1|\le t)>0$. Furthermore,
$X$ and $M$ are positive recurrent if and only if 
$\e[\log(1+|\xi_1|)]$ is finite.

If the innovations $\{\xi_n\}$ 
take only positive values, then we may define 
$$
\eta_n:=\log_A\xi_n
\quad\text{and}\quad 
R_n:=\log_A M_n,
$$
where $A=a^{-1}$. 
Then the sequence $R=\{R_n\}_{n\ge 0}$ satisfies the recursive relation 
$$
R_n=\max\{R_{n-1}-1,\eta_n\},\quad n\ge1. 
$$
This Markov  chain is a special random exchange process, see Helland and Nilsen~\cite{HN76} for the definition of this class of processes.

In this paper we shall consider the case when the tail of innovations decreases logarithmically. More precisely, the main part of the paper will deal with situation when
\begin{equation}
\label{eq:log.tail}
\pr(\xi_1>x)\sim\frac{d}{\log x}\quad 
\text{as }x\to\infty
\end{equation}
with some constant $d>0$. This is equivalent to 
\begin{equation}
\label{eq:log.tail.2}
\pr(\eta_1>y)\sim\frac{c}{y},\quad c:=\frac{d}{\log A}.
\end{equation}
(We shall explicitly mention one of these two conditions every time we need it.)

Notice also that if $d>0$ then $\e\log(1+\xi_1)=\infty$ and, consequently, 
the chains $\{X_n\}_{n\ge 0}$ and $\{M_n\}_{n\ge 0}$ are not positive recurrent. If \eqref{eq:log.tail} holds then, using the criterion \eqref{eq:recurrence}, we conclude that 
\begin{align*}
&\ \bullet d>\log A\ (c>1)\quad \Rightarrow\quad  
\{X_n\}_{n\ge 0}\text{ and }\{M_n\}_{n\ge 0}\text{ are transient;}\\
&\ \bullet d<\log A\ (c<1)\quad \Rightarrow\quad  
\{X_n\}_{n\ge 0}\text{ and }\{M_n\}_{n\ge 0}\text{ are null-recurrent.}
\end{align*}
In the critical case $d=\log A$ ($c=1$) one has to consider further terms in the asymptotic representation for tails
$\pr(\xi_1>x)$, $\pr(\eta_1>y)$.
Assume that, for some $k\ge0$, 
$$
\pr(\eta_1>y)=\sum_{j=0}^k\frac{1}{y}
\prod_{l=1}^j\frac{1}{\log_{(l)}y}
+(r_k+o(1))\frac{1}{y}\prod_{l=1}^{k+1}\frac{1}{\log_{(l)}y},
\quad y\to\infty,
$$
where $\log_{(l)}x$ is the $l$-th iteration of the logarithm. 
Then, applying \eqref{eq:recurrence} once again, we obtain 
\begin{align*}
&\ \bullet r_k>1\quad \Rightarrow\quad  
\{X_n\}\text{ and }\{M_n\}\text{ are transient;}\\
&\ \bullet r_k<1\quad \Rightarrow\quad  
\{X_n\}\text{ and }\{M_n\}\text{ are null-recurrent.}
\end{align*}

A further similarity between the chains $\{X_n\}_{n\ge 0}$ and $\{M_n\}_{n\ge 0}$ consists in the joint scaling behaviour of these chains. More precisely, 
Buraczewski and Iksanov ~\cite{BurIks15} have shown that if \eqref{eq:log.tail} is valid then 
\begin{equation}
\label{eq:weak.limit.1}
\left(\frac{\log_A X_{nt}}{n}\right)_{t\ge 0}
\Rightarrow Z=(Z_t)_{t\ge 0}
\end{equation}
in the Skorohod $J_1$-topology on the space $D$. 
The limiting process $Z$ is a self-similar Markov process. 
In \cite{BurIks15} it is described with the help of an appropriate Poisson point process.
One can describe this limiting process also via the transition probabilities:
\begin{equation}
\label{eq:trans.prob}
\pr_x((x-t)^+\le Z_t\le y)=\left(\frac{y}{y+t}\right)^c,
\quad y\ge (x-t)^+,\ x\ge0.
\end{equation}

It is easy to see that if $X_0=M_0$ then
$$
M_k\le X_k\le (k+1)M_k\quad\text{for all }k\ge1.  
$$
This implies that \eqref{eq:weak.limit.1} is equivalent to 
\begin{equation}
\label{eq:weak.limit.2}
\left(
\frac{\log_A M_{nt}}{n}
\right)
\Rightarrow Z.
\end{equation}
In its turn, \eqref{eq:weak.limit.2} is equivalent to 
\begin{equation}
\label{eq:weak.limit.3}
\left(
\frac{R_{nt}}{n}
\right)_{t\ge 0}
\Rightarrow Z.
\end{equation}

The main purpose of this paper is to study the asymptotic behaviour of recurrence times 
\begin{align*}
&T_x^{(X)}:=\inf\{k\ge1:\ X_k\le x\},\\
&T_x^{(M)}:=\inf\{k\ge1:\ M_k\le x\},\\
&T_x^{(R)}:=\inf\{k\ge1:\ R_k\le x\}.
\end{align*}

Persistence of auto-regressive processes has attracted a significant attention of many researchers in the recent past, 
but almost all results known in the literature   deal with the case when some power moments
of the innovations $\xi_k$ are finite. It is known that the tail of $T_x^{(X)}$ decreases exponentially fast
$$
-\frac{1}{n}\log\pr(T_x^{(X)}>n)\to\lambda\in(0,\infty)
$$
see \cite{AMZ21}, \cite{HKW19} and references there. If all power moments of innovations are finite then $\pr(T_x^{(X)}>n)\sim Ce^{–\lambda n}$ and the conditional distribution $\pr(X_n\in\cdot|T_x^{(X)}>n)$ converges towards the corresponding quasi-stationary distribution, see \cite{HKW19}.
It is worth mentioning that one can compute the persistence exponent
$\lambda$ in some special cases only. Some examples of autoregressive processes, for which there exist  closed form expressions for
$\lambda$, can be found in \cite{ABRS} and in \cite{AMZ21}. The authors of \cite{AK19} have found a series representation for $\lambda$ in the case of normally distributed innovations. 

In the present paper we concentrate on the case when all power moments of innovations are infinite. This corresponds, as we shall show, to a subexponential decay of the tail of $T_x^{(X)}$. 

We start with the null-recurrent case. More precisely we consider first 
the innovations which satisfy \eqref{eq:log.tail}. 
As we have mentioned before, the chains $\{X_n\}_{n\ge 0}$, 
$\{M_n\}_{n\ge 0}$ and $\{R_n\}_{n\ge 0}$ have the same 
scaling limit $Z$ in this case. For that reason we first collect some crucial for us properties of the process $Z$.
\begin{theorem}
\label{thm:Z} 
(a) If $c\le1$ then the process $Z$ is recurrent. If $c<1$ then the stopping time $T_0^{(Z)}:=\inf\{s>0: Z_s=0\text{ or }Z_{s-}=0\}$ is almost surely finite and, furthermore,
\begin{equation}
\label{eq:tau_tail}
\pr\left(T_0^{(Z)}>t\big|Z_0=z\right)=
\left\{ 
\begin{array}{cl}
1, &t<z\\
\frac{1}{B(c,1-c)}\int_0^{z/t}(1-u)^{c-1}u^{-c}du, &t\ge z.
\end{array}
\right.
\end{equation}
(b) The function $u(z)=z^{1-c}$ is harmonic for $Z$ killed at $T_0^{(Z)}$:
$$
u(z)=\e_z[u(Z_t);T_0^{(Z)}>t],\quad t,z>0.
$$
(c) The sequence of distributions $\pr_z\left(Z\in\cdot|T_0^{(Z)}>1\right)$ on $D[0,1]$ converges weakly, as $z\to0$, towards a non-degenerate distribution $\mathbf{Q}$.
\end{theorem}

We now turn to the recurrence times of the chains $\{M_n\}_{n\ge 0}$ and $\{R_n\}_{n\ge 0}$. Since $R_n=\log_A M_n$,
$$
T^{(R)}_{x_0}=\inf\{n\ge 1: R_n\le x_0\}
=\inf\{n\ge 1: M_n\le A^{x_0}\}=T^{(M)}_{A^{x_0}}.
$$
Thus, it suffices to formulate the results for one of these processes.

Set 
$$
u_0(x):=\int_0^x\pr(\eta_1>y)dy,\quad x\ge0
$$
and 
\begin{equation}
\label{eq:U.def}
U_0(x):=\int_0^xe^{-u_0(y)}dy,\quad x\ge0.
\end{equation}
If \eqref{eq:log.tail.2} holds then $u_0(x)\sim c\log x$ as $x\to\infty$
and $e^{-u_0(x)}$ is regularly varying of index $-c$. Consequently, the function $U_0(x)$ is regularly varying of index $1-c$.
\begin{theorem}
\label{thm:Rn}
Assume that $x_0$ is such that $\pr(\eta_1\le x_0)\pr(\eta_1>x_0)>0$.
Then the equation 
$$
G(x)=\e_x[G(R_1);T^{(R)}_{x_0}>1],\quad x>x_0
$$
has a non-trivial solution if and only if $\e\eta_1^+=\infty$.
In the latter case 
$$
G(x)=C\left(1+\sum_{j=1}^\infty\prod_{k=0}^{j-1}\pr(\eta_1\le x_0+j)
{\rm 1}_{(x_0+j+1,\infty)}(x)\right)
$$
for every $C\in\mathbb{R}$.\\
If \eqref{eq:log.tail.2} holds with some $c\in(0,1)$ then
\begin{itemize}
 \item[(i)]  $G(x)\sim \gamma U_0(x)$ for some $\gamma\in(0,\infty)$;
 \item[(ii)] there exists a constant $C>0$ such that
 $$
\frac{1}{C}\frac{G(x\wedge n)}{G(n)}
\le\pr_x(T^{(R)}_{x_0}>n)
\le C\frac{G(x)}{G(n)},\quad n\ge1,\ x>x_0;
$$
\item[(iii)] there exists a positive constant $\varkappa=\varkappa(c)$ such that 
$$
\pr_x(T^{(R)}_{x_0}>n)\sim\varkappa\frac{G(x)}{G(n)},
\quad n\to\infty.
$$
and the sequence of conditional distributions
$\pr_x\left(\frac{R_{[nt]}}{n}\in\cdot\Big|T^{(R)}_{x_0}>n\right)$
on $D[0,1]$ converges weakly to $\mathbf{Q}$ defined in Theorem~\ref{thm:Z}.
\end{itemize}
\end{theorem}

We now state our main result for the chain $\{X_n\}_{n\ge 0}$.
\begin{theorem}
\label{thm:tail_of_T}
Assume that \eqref{eq:log.tail.2} holds with some $c\in(0,1)$.
(This is equivalent to \eqref{eq:log.tail} with $0<d<\log A$.)
For every $x_0$ satisfying $\pr(ax_0+\xi_1\le x_0)>0$ we have:
\begin{itemize}
 \item[(i)] There exists a strictly positive on $(x_0,\infty)$
function $V$ such that 
$$
V(x)=\e_x[V(X_1);T^{(X)}_{x_0}>1], \quad x>x_0.
$$
In other words, $V$ is harmonic for the chain $\{X_n\}$ killed at leaving $(x_0,\infty)$. Furthermore, $V(A^x)\sim U_0(x)$, where $U_0$ is defined in \eqref{eq:U.def}.
\item[(ii)] There exists a constant $C$ such that 
\begin{equation}
\label{eq:tail_bounds}
\frac{1}{C}\frac{V(x\wedge A^n)}{V(A^n)}\le\pr_x(T^{(X)}_{x_0}>n)
\le C\frac{V(x)}{V(A^n)}
\end{equation}
for all $n\ge1$ and all $x>x_0.$
\item[(iii)] There exists a positive constant
$\varkappa=\varkappa(c)$ such that, for every $x>x_0$,
\begin{equation}
\label{eq:tail_asymp}
\pr_x(T^{(X)}_{x_0}>n)\sim\varkappa\frac{V(x)}{V(A^n)},
\quad n\to\infty.
\end{equation}
Furthermore, the sequence of conditional distributions
$$
\pr_x\left(\frac{\log_A X_{[nt]}}{n}\in\cdot\Big|T^{(X)}_{x_0}>n\right)
$$
on $D[0,1]$ converges weakly to $\mathbf{Q}$ defined in Theorem~\ref{thm:Z}.
\end{itemize}
\end{theorem}

We now turn to the positive recurrent case: $\e[\eta_1]<\infty$. To determine the tail behaviour of recurrence times we shall assume that 
$\overline F(y):=\pr(\eta_1>y) $ is subexponential. 
We make use of the following class introduced in~\cite{Kl88}.
\begin{definition}\label{sstar}
    A distribution function $F$ with  finite $\mu_+=\int_0^\infty \overline F(y)dy<\infty$ belongs to the class $\mathcal S^*$ of strong subexponential distributions  if 
 $\overline F(x)>0$ 
for all $x$ and 
\[
\frac{\int_0^x \overline F(x-y)\overline F(y)dy}{\overline F(x)}
\to 2 \mu_+,\quad \mbox{ as } x\to \infty. 
\]
\end{definition}
This class is a proper subclass of class $\mathcal S$ of subexponential distributions. It is shown 
in~\cite{Kl88} that the Pareto, lognormal and Weibull distributions belong to the class $\mathcal S^*$. An example of a subexponential distribution with finite mean which does not belong to $\mathcal S^*$ can be found in~\cite{DFK04}. 

\begin{theorem}\label{thm:subex_for_R}
Assume that $x_0$ is such that $\pr(\eta_1\le x_0)\pr(\eta_1>x_0)>0$. 
Assume also that $\e\eta_1<\infty$ and that $F\in \mathcal S^*$. Then, for any $x>x_0$
\begin{equation} 
\pr_x(T^{(R)}_{x_0}>n)  \sim \e_x[T^{(R)}_{x_0}]\pr(\eta_1>n). 
\end{equation}
The expectation $\e_x[T^{(R)}_{x_0}]$ can be computed explicitly:
for every $n\ge0$ and every $x\in(x_0+n,x_0+n+1]$ one has
$$
\e_x[T^{(R)}_{x_0}]=
\frac{1}{\prod_{k=0}^{\infty}\pr(\eta_1\le x_0+k)}
\left(1+\sum_{j=1}^{n}\prod_{k=0}^{j-1}\pr(\eta_1\le x_0+k)\right).
$$
\end{theorem}   
Our approach to the proof of this theorem is based on a recursive equation for the tail of $T^{(R)}_{x_0}$, see Proposition~\ref{prop:tail.recursion} below. 
In the case of the chain $\{X_n\}_{n\ge 0}$ we do not have such an equation and we have to work with upper and lower estimates. This leads to more restrictive assumptions on the tail of innovations $\eta_k$. 

\begin{theorem}\label{thm:subex_for_X}
Assume that $x_0$ is such that $\pr(ax_0+\xi_1\le x_0)>0$. 
Assume also that $\e\eta<\infty$,that $F\in \mathcal S^*$ and that 
\begin{equation}
\label{eq:log-insens}
\pr(\eta>x)\sim\pr(\eta>x-\log x),\quad\text{as }x\to\infty.
\end{equation}
Then, for any $x>x_0$,
\begin{equation} 
\pr_x(T^{(X)}_{x_0}>n)  \sim \e_x[T^{(X)}_{x_0}]\pr(\eta>n). 
\end{equation}
\end{theorem}    
The rest of the paper is organised as follows. 
In Section~\ref{sec:Z} 
we discuss properties of $Z_t$ and prove Theorem~\ref{thm:Z}. 
In Section~\ref{sec:harmonic} we construct 
harmonic functions for processes under consideration 
proving corresponding parts of Theorem~\ref{thm:Rn} and Theorem~\ref{thm:tail_of_T}. 
In Section~\ref{sec:l.u.bounds} we 
derive lower and upper bounds for recurrence times Theorem~\ref{thm:Rn} and Theorem~\ref{thm:tail_of_T}. 
proving part (ii) of  Theorem~\ref{thm:Rn} and Theorem~\ref{thm:tail_of_T}. 
In Section~\ref{sec:asym} we obtain 
the asymptotics for tails of recurrence times given  
in  part (iii) of  Theorem~\ref{thm:Rn} and Theorem~\ref{thm:tail_of_T}.  
In Section~\ref{sec:subex_for_R} we prove 
Theorem~\ref{thm:subex_for_R} and 
in Section~\ref{sec:subex_for_X} we prove 
Theorem~\ref{thm:subex_for_X}.

\section{Properties of the limiting process $Z$: proof of Theorem~\ref{thm:Z}.}
\label{sec:Z}
It follows from \eqref{eq:trans.prob} that if $t<x$ then 
\begin{equation}
\label{eq:trans.prob1}
\pr_x(Z_t=x-t)=\left(\frac{x-t}{x}\right)^c
\ \text{and}\  
\frac{\pr_x(Z_t\in dy)}{dy}
=\frac{ct y^{c-1}}{(t+y)^{c+1}},\ y>x-t.
\end{equation}
If $t\ge x$ then 
\begin{equation}
\label{eq:trans.prob2}
\frac{\pr_x(Z_t\in dy)}{dy}
=\frac{ct y^{c-1}}{(t+y)^{c+1}},\ y>0.
\end{equation}
It is immediate from \eqref{eq:trans.prob} that if $c\le1$ then 
$$
\int_0^\infty\pr_x(Z_t\le y)dt=\infty
$$
for all $x,y>0$.  Therefore, the process $Z$ is recurrent: it spends infinite amount of time in every interval $[0,y]$. 

We next show that the state $0$ is recurrent in the case $c<1$. More precisely, we show that $\pr_z(\tau_0^{(Z)}<\infty)=1$ for every $z>0.$
For that reason we compute first the generator of $Z$. 
Fix some $x>0$ and a continuously differentiable 
bounded  function $f$. It follows then from 
\eqref{eq:trans.prob1} that 
$$
\e_x[f(Z_t)]
=f(x-t)\left(\frac{x-t}{x}\right)^c
+ct\int_{x-t}^\infty\frac{y^{c-1}}{(y+t)^{c+1}}f(y)dy,
\quad t<x.
$$
Therefore,
$$
\frac{\e_x[f(Z_t)]-f(x)}{t}
=\frac{f(x-t)-f(x)}{t}
+f(x-t)\frac{\left(\frac{x-t}{x}\right)^c-1}{t}
+c\int_{x-t}^\infty\frac{y^{c-1}}{(y+t)^{c+1}}f(y)dy.
$$
Letting now $t\to0$, we conclude that 
\begin{align}
\label{eq:gen1}
\nonumber
\mathcal{L}f(x)
&=-f'(x)-c\frac{f(x)}{x}+c\int_x^\infty\frac{f(y)}{y^2}dy\\
&=-f'(x)+c\int_x^\infty\frac{f(y)-f(x)}{y^2}dy,
\quad x>0.
\end{align}
It is easy to see that this generator can be represented as follows 
\begin{align*}
&\mathcal{L}f(x)\\
&\hspace{2mm}=-\left(1-c\int_1^\infty\frac{\log u}{1+\log^2u}\frac{du}{u^2}\right)f'(x)\\
&\hspace{2cm}+\frac{c}{x}\int_1^\infty\left(f(ux)-f(x)-
\frac{\log u}{1+\log^2u}xf'(x)\right)\frac{du}{u^2}\\
&\hspace{2mm}=-\left(1-c\int_1^\infty\frac{\log u}{1+\log^2u}\frac{du}{u^2}\right)f'(x)+\frac{1}{x}\int_1^\infty h^*(x,u)\frac{c\log^2u}{u^2(1+\log^2u)}du,
\end{align*}
where 
$$
h^*(x,u)=\left(f(ux)-f(x)-\frac{\log u}{1+\log^2u}xf'(x)\right)
\frac{1+\log^2u}{\log^2u}.
$$
Then, according to Theorem 6.1 in Lamperti~\cite{Lamp72},
$\{Z_t,t<\tau^{(Z)}_0\}$ can be represented as the exponential functional of a time-changed L\'evy process with the following L\'evy-Khintchine exponent:
$$
\Psi(\lambda)=-i\lambda
\left(1-c\int_1^\infty\frac{\log u}{1+\log^2u}\frac{du}{u^2}\right)
+\int_0^\infty\left(e^{i\lambda y}-1-\frac{i\lambda y}{1+y^2}\right)ce^{-y}dy.
$$
Simplifying this expression, we get 
$$
\Psi(\lambda)=-i\lambda+\int_0^\infty(e^{i\lambda y}-1)e^{-y}cdy.
$$
This corresponds to the process $\zeta_t-t$, where $(\zeta_t)_{t\ge 0}$ is a compound Poisson process with intensity $c$ and with exponentially distributed jumps. 
In particular, $\zeta_t-t\to-\infty$ a.s. as $t\to\infty$ in the case $c<1$ and $\zeta_t-t$ is oscillating in the case $c=1$. Then $T^{(Z)}_0$ is finite almost surely iff $c<1$.

To prove \eqref{eq:tau_tail} we define 
$$
g(t,z):=\pr_z(T_0^{(Z)}>t).
$$
It is clear that $g(t,z)=1$ for all $t\le z$. Using \eqref{eq:trans.prob2}, we see that $g$ solves the equation 
$$
g(t,z)=g(t-s,z-s)\left(\frac{z-s}{z}\right)^c
+cs\int_{z-s}^\infty\frac{y^{c-1}}{(y+s)^{c+1}}g(t-s,y)dy,\quad s<z.
$$
Letting $s\to0$ we obtain the following decomposition for the expression on the right hand side:
$$
g(t-s,z-s)\left(1-\frac{cs}{z}\right)+
cs\int_{z}^\infty\frac{g(t,y)}{y^{2}}dy+o(s).
$$
Therefore,
$$
\lim_{s\to0}\frac{g(t,z)-g(t-s,z-s)}{s}
=-\frac{c}{z}g(t,z)+
c\int_{z}^\infty\frac{g(t,y)}{y^{2}}dy.
$$
As a result we have the following differential equation
\begin{equation}
\label{eq:diff.eq}
\frac{\partial}{\partial t}g(t,z)
+\frac{\partial}{\partial z}g(t,z)
=-\frac{c}{z}g(t,z)+
c\int_{z}^\infty\frac{g(t,y)}{y^{2}}dy,
\quad t>z.
\end{equation}
Since the process $Z$ is self-similar with index $1$,
$$
g(t,z)=\pr(T_0^{(Z)}>t|Z_0=z)=\pr(T_0^{(Z)}>t/z|Z_0=1)
=g\left(\frac{t}{z},1\right)=:h\left(\frac{t}{z}\right).
$$
It follows then from \eqref{eq:diff.eq} that the function $h$
satisfies 
$$
\frac{1}{z}h'\left(\frac{t}{z}\right)
-\frac{t}{z^2}h'\left(\frac{t}{z}\right)
=-\frac{c}{z}h\left(\frac{t}{z}\right)
+c\int_{z}^\infty\frac{h(t/y)}{y^{2}}dy,\  t>z.
$$
Noting that $h(r)=1$ for all $r\le 1$ and substituting $t/y=x$, we get
\begin{align*}
\int_{z}^\infty\frac{h(t/y)}{y^{2}}dy
&=\int_{z}^t\frac{h(t/y)}{y^{2}}dy+\frac{1}{t}\\
&=\frac{1}{t}\int_1^{t/z}h(x)dx+\frac{1}{t}.
\end{align*}
Therefore,
$$
(1-y)h'(y)=-ch(y)+\frac{c}{y}\left(1+\int_1^y h(x)dx\right),
\quad y>1.
$$
Differentiating this equation, we get 
\begin{align*}
(1-y)h''(y)-h'(y)
&=-ch'(y)+\frac{c}{y}h(y)
-\frac{c}{y^2}\left(1+\int_1^y h(x)dx\right)\\
&=-ch'(y)+\frac{c}{y}h(y)
-\frac{1}{y}((1-y)h'(y)+ch(y)).
\end{align*}
Rearranging the terms, we arrive at the equation 
$$
(1-y)h''(y)=\left(1-c-\frac{1-y}{y}\right)h'(y).
$$
This is equivalent to 
$$
(\log h'(y))'=\frac{h''(y)}{h'(y)}=\frac{c-1}{y-1}-\frac{1}{y}.
$$
Consequently,
$$
h'(y)=C(y-1)^{c-1}y^{-1}
\quad\text{and}\quad
h(y)=C\int_x^\infty (y-1)^{c-1}y^{-1}dy.
$$
The boundary condition $h(1)=1$ leads to the equality 
$$
h(x)=\frac{\int_x^\infty (y-1)^{c-1}y^{-1}dy}
{\int_1^\infty (y-1)^{c-1}y^{-1}dy},\quad x\ge1.
$$
Substituting in these integrals $y=1/u$, we finally get 
$$
h(x)=\frac{1}{B(c,1-c)}\int_0^{1/x}(1-z)^{c-1}z^{-c}dz,
\quad x\ge1.
$$
As a result we have \eqref{eq:tau_tail}.
This formula can be also obtained via the Lamperti transformation mentioned above. If $Z_0=1$ then $T_0^{(Z)}$ has the same  distribution as
$I:=\int_0^{\infty}e^{\zeta_t}dt$ and $1/I$ has the beta distribution with parameters $c$ and $1-c$, see Bertoin and Yor~\cite{BerYor05}. 

We now turn to the proof of part (b).
We start by computing the expectation 
$\e_x[Z_t^{1-c}]$. If $t\le x$ then, in view of \eqref{eq:trans.prob1},
\begin{align*}
\e_x[Z_t^{1-c}]
&=\int_{x-t}^\infty y^{1-c}\pr_x(Z_t\in dy)\\
&=(x-t)^{1-c}\left(\frac{x-t}{x}\right)^c
+\int_{x-t}^\infty\frac{ct}{(t+y)^{c+1}}dy\\
&=\frac{x-t}{x^c}+ct\int_x^\infty\frac{dy}{y^{c+1}}=x^{1-c}.
\end{align*}
If $t>x$ then, by \eqref{eq:trans.prob2},
$$
\e_x[Z_t^{1-c}]
=\int_{0}^\infty\frac{ct}{(t+y)^{c+1}}dy=t^{1-c}.
$$
Using these equalities, we obtain 
\begin{align}
\label{eq:Z_harm}
\nonumber
\e_x[Z_t^{1-c};T_0^{(Z)}>t]
&=\e_x[Z_t^{1-c}]-\e_x[Z_t^{1-c};T_0^{(Z)}\le t]\\
\nonumber
&=(\max\{t,x\})^{1-c}
-\int_0^t \pr_x(T_0^{(Z)}\in ds)\e_0[Z_{t-s}^{1-c}]\\
&=(\max\{t,x\})^{1-c}
-\int_0^t (t-s)^{1-c}\pr_x(T_0^{(Z)}\in ds).
\end{align}
It follows from \eqref{eq:tau_tail} that 
the integral in \eqref{eq:Z_harm} is zero for $t\le x$, and that for 
$t>x$ one has 
\begin{align*}
&\int_0^t (t-s)^{1-c}\pr_x(T_0^{(Z)}\in ds)\\
&\hspace{1cm}=\int_x^t (t-s)^{1-c}\pr_x(T_0^{(Z)}\in ds)\\
&\hspace{1cm}=\frac{1}{B(c,1-c)}\int_x^t (t-s)^{1-c}
\left(1-\frac{x}{s}\right)^{c-1}\left(\frac{x}{s}\right)^{-c}
\frac{x}{s^2}ds\\
&\hspace{1cm}=\frac{1}{B(c,1-c)}\int_x^t (t-s)^{1-c}
\left(1-\frac{x}{s}\right)^{c-1}\left(\frac{s}{x}\right)^{c-1}
\frac{1}{s}ds\\
&\hspace{1cm}=\frac{x^{1-c}}{B(c,1-c)}\int_x^t (t-s)^{1-c}
\left(s-x\right)^{c-1}\frac{1}{s}ds.
\end{align*}
With the help of the substitution
$v=\left(\frac{s-x}{t-s}\right)$ we get 
\begin{align*}
\int_x^t (t-s)^{1-c}\left(s-x\right)^{c-1}\frac{1}{s}ds
&=\int_0^\infty v^{c-1}\frac{1+v}{x+tv}
\left(\frac{t}{1+v}-\frac{x+tv}{(1+v)^2}\right)dv\\
&=t\int_0^\infty\frac{v^{c-1}}{x+tv}dv
-\int_0^\infty\frac{v^{c-1}}{1+v}dv\\
&=\left(\left(\frac{t}{x}\right)^{1-c}-1\right)
\int_0^\infty\frac{v^{c-1}}{1+v}dv.
\end{align*}
Noting now that $\int_0^\infty\frac{v^{c-1}}{1+v}dv=B(c,1-c)$, we conclude that
$$
\int_0^t (t-s)^{1-c}\pr_x(T_0^{(Z)}\in ds)
=\max\{t^{1-c}-x^{1-c},0\}.
$$
Plugging this into \eqref{eq:Z_harm}, we conclude that 
$$
\e_x[Z_t^{1-c};T_0^{(Z)}>t]=x^{1-c}
$$
for all $x,t>0$. Thus, (b) is proven.

To prove (c) we first consider one-dimensional marginals.
For $t\le x$ one has 
$$
\pr_x(Z_t\le y;T_0^{(Z)}>t)=\pr_x(Z_t\le y), \ y>0.
$$
If $t>x$ then 
\begin{align*}
\pr_x(Z_t\le y;T_0^{(Z)}>t)
&=\pr_x(Z_t\le y)-\pr_x(Z_t\le y;T_0^{(Z)}\le t)\\
&=\pr_x(Z_t\le y)-\int_x^t\pr_x(T_0^{(Z)}\in ds)
\pr_0(Z_{t-s}\le y).
\end{align*}
Using now \eqref{eq:trans.prob} and \eqref{eq:tau_tail}, we get 
\begin{align*}
&\pr_x(Z_t\le y;T_0^{(Z)}>t)\\
&\hspace{1cm}=\left(\frac{y}{y+t}\right)^c 
-\frac{1}{B(c,1-c)}\int_x^t\left(\frac{y}{y+t-s}\right)^c
\left(1-\frac{x}{s}\right)^{c-1}\left(\frac{x}{s}\right)^{-c}
\frac{x}{s^2}ds\\
&\hspace{1cm}=\left(\frac{y}{y+t}\right)^c 
-\frac{1}{B(c,1-c)}\int_x^t\left(\frac{y}{y+t-s}\right)^c
\left(\frac{s}{x}-1\right)^{c-1}\frac{1}{s}ds.
\end{align*}
This representation can be used to obtain an exact formula for 
the transition kernel $\pr_x(Z_t\le y;T_0^{(Z)}>t)$ in terms of the hypergeometric function of two variables. 
Instead of doing that we shall determine the asymptotic, as $x\to0$, behaviour of the distribution function $\pr_x(Z_t\le y;T_0^{(Z)}>t)$. We start by noting that 
\begin{align}
\label{eq:dist_repr}
\nonumber
&\pr_x(Z_t\le y;T_0^{(Z)}>t)\\
&\hspace{1cm}=\left(\frac{y}{y+t}\right)^c\pr_x(T_0^{(Z)}>t) 
-\frac{1}{B(c,1-c)}\int_x^t\Delta_{y,t}(s)
\left(\frac{s}{x}-1\right)^{c-1}\frac{1}{s}ds,
\end{align}
where 
$$
\Delta_{y,t}(s)=\left(\frac{y}{y+t-s}\right)^c
-\left(\frac{y}{y+t}\right)^c.
$$
Fix some $\varepsilon>0$. It is easy to see that 
$$
\Delta_{y,t}(s)=\frac{y^c}{(y+t)^c}
\left[\left(1+\frac{s}{t+y-s}\right)^c-1\right]
\le \frac{cy^c}{(y+t)^c}\frac{s}{y+t-\varepsilon}
$$
for all $s\le\varepsilon$. Therefore, for all $x<\varepsilon$,
\begin{align}
\label{eq:eps_bound.1}
\nonumber
\int_x^\varepsilon\Delta_{y,t}(s)
\left(\frac{s}{x}-1\right)^{c-1}\frac{1}{s}ds
&\le \frac{cy^c}{(y+t-\varepsilon)(y+t)^c}
\int_x^\varepsilon\left(\frac{s}{x}-1\right)^{c-1}ds\\
& \le \frac{y^c}{(y+t-\varepsilon)(y+t)^c}x^{1-c}\varepsilon^c.
\end{align}
Furthermore, as $x\to0$, 
\begin{align*}
\int_\varepsilon^t\Delta_{y,t}(s)
\left(\frac{s}{x}-1\right)^{c-1}\frac{1}{s}ds
&=x^{1-c}\int_\varepsilon^t\Delta_{y,t}(s)
\left(s-x\right)^{c-1}\frac{1}{s}ds\\
& =x^{1-c}(1+o(1))\int_\varepsilon^t\Delta_{y,t}(s)
s^{c-2}ds.
\end{align*}
Combining this with \eqref{eq:eps_bound.1} and letting $\varepsilon\to0$, we conclude that 
\begin{equation}
\label{eq:int_lim}
\lim_{x\to0} x^{c-1}\int_x^t\Delta_{y,t}(s)
\left(\frac{s}{x}-1\right)^{c-1}\frac{1}{s}ds
=\int_0^t\Delta_{y,t}(s)s^{c-2}ds.
\end{equation}
Using the equality 
$$
\Delta_{y,t}(s)=\int_{y/(t+y)}^{y/(t+y-s)}cu^{c-1}du
$$
and the Fubini theorem, we have 
\begin{align*}
\int_0^t\Delta_{y,t}(s)s^{c-2}ds
&=\int_0^t\left(\int_{y/(t+y)}^{y/(t+y-s)}cu^{c-1}du\right)s^{c-2}ds\\
&=\int_{y/(y+t)}^1cu^{c-1}\left(\int_{y+t-y/u}^ts^{c-2}ds\right)du\\
&=\frac{c}{1-c}\int _{y/(y+t)}^1cu^{c-1}
\left((y+t-y/u)^{c-1}-t^{c-1}\right)du\\
&=\frac{c}{1-c}\int _{y/(y+t)}^1 ((y+t)u-y)^{c-1}du
-\frac{c}{1-c}\int _{y/(y+t)}^1 u^{c-1}du\\
&=\frac{1}{1-c}\frac{t^c}{y+t}
-\frac{1}{1-c}t^{c-1}\left(1-\left(\frac{y}{y+t}\right)^c\right).
\end{align*}
Combining this with \eqref{eq:int_lim} and noting that 
\begin{align}
\label{eq:tail_small_x}
\pr_x(T_0^{(Z)}>t)\sim\frac{x^{1-c}}{(1-c)B(c,1-c)}t^{c-1}, 
\quad x\to0,
\end{align}
we conclude that 
$$
\lim_{x\to0}\frac{\int_x^t\Delta_{y,t}(s)
\left(\frac{s}{x}-1\right)^{c-1}\frac{1}{s}ds}
{\pr_x(T_0^{(Z)}>t)}
=B(c,1-c)\left[\left(\frac{y}{y+t}\right)^c-\frac{y}{y+t}\right].
$$
Combining this with \eqref{eq:dist_repr}, we finally obtain 
\begin{equation}
\label{eq:cond.lim.Z}
\lim_{x\to0}\pr_x(Z_t\le y|T_0^{(Z)}>t)=
\frac{y}{y+t},\quad y>0.
\end{equation}

Using the harmonic function $u(x)=x^{1-c}$ we now define the Doob $h$-transform of $\mathcal{L}$:
$$
\widehat{\mathcal{L}}f(x)
:=\frac{1}{u(x)}\mathcal{L}(uf)(x),\quad x>0.
$$
The corresponding probability measure is given by 
$$
\widehat{\e}_x[g(Z)]
:=\frac{1}{u(x)}\e_x[g(Z)u(Z_t);\tau_0^{(Z)}>t]
$$
for every bounded measurable functional $g$ on $D[0,t]$.

From \eqref{eq:gen1} we infer that 
\begin{align}
\label{eq:gen2}
\nonumber
\widehat{\mathcal{L}}f(x)
&=\frac{1}{u(x)}\left[-u(x)f'(x)-u'(x)f(x)-c\frac{u(x)f(x)}{x}
+c\int_x^\infty\frac{u(y)f(y)}{y^2}dy\right]\\
\nonumber
&=-f'(x)-\frac{f(x)}{x}
+\frac{c}{x^{1-c}}\int_x^\infty\frac{f(y)}{y^{1+c}}dy\\
&=-f'(x)+\frac{c}{x^{1-c}}\int_x^\infty\frac{f(y)-f(x)}{y^{1+c}}dy.
\end{align}
As a result we have the following representation:
\begin{align*}
\widehat{\mathcal{L}}f(x)
&=-f'(x)
+\frac{c}{x}\int_1^\infty\frac{f(ux)-f(x)}{u^{1+c}}du\\
&=-\left(1-c\int_1^\infty\frac{\log u}{1+\log^2u}\frac{du}{u^2}\right)
f'(x)
+\frac{1}{x}\int_1^\infty h^*(x,u)\frac{c\log^2u}{u^2(1+\log^2u)}du. 
\end{align*}
This implies that, under $\widehat{\pr}$, $Z$ is self-similar and can be expressed via a L\'evy process with the 
characteristic exponent 
$$
\widehat{\Psi}(\lambda)
=-i\lambda+\int_0^\infty(e^{i\lambda y}-1)e^{-cy}cdy.
$$
This corresponds to $\widehat{\zeta}_t-t$, where 
$(\widehat{\zeta}_t)_{t\ge 0}$ is a compound Poisson process with intensity $c$ and with positive jumps, which have exponential with parameter $c$ distribution. This L\'evy process is clearly oscillating. Consequently, 
$$
\widehat{\pr}_x(T_0^{(Z)}=\infty)=1,\quad x>0. 
$$

According to Theorem 2 in Caballero and Chaumont~\cite{CC06}, the sequence of measures $\widehat{\pr}_x$ converges weakly
on $D[0,1]$, as $x\to0$, to a non-degenerate probabilistic measure $\widehat{\pr}_0$. 
We now show that this implies that
$\pr_x\left(Z\in\cdot\ |T_0^{(Z)}>1\right)$ also converges weakly on $D[0,1]$. 

It follows from the definition of $\widehat{\pr}_x$ that 
\begin{align*}
\widehat{\pr}_0(Z_1\le y)
&=\lim_{x\to0}\widehat{\pr}_x(Z_1\le y)\\
&=\lim_{x\to0}\frac{\pr_x(T_0^{(Z)}>1)}{u(x)}
\e_x[u(Z_1){\rm 1}\{Z_1\le y\}|T_0^{(Z)}>1].
\end{align*}
Applying now \eqref{eq:tail_small_x} and \eqref{eq:cond.lim.Z}, we obtain 
\begin{align*}
\widehat{\pr}_0(Z_1\le y)
=\frac{1}{(1-c)B(c,1-c)}\int_0^y\frac{z^{1-c}}{(1+z)^2}dz.
\end{align*}
Consequently, the density of $Z_1$ under $\widehat{\pr}_0$
is proportional to $\frac{z^{1-c}}{(1+z)^2}$.
Let $g$ be a bounded and continuous functional on $D[0,1]$
and let $\varepsilon$ be a fixed positive number. Since 
$\widehat{\pr}_0(Z_1=\varepsilon)=0$, the weak convergence
$\widehat{\pr}_x\Rightarrow \widehat{\pr}_0$ implies that
\begin{align}
\label{eq:step.1}
\lim_{x\to0}
\widehat{\e}_x\left[\frac{g(Z)}{u(Z_1)};Z_1>\varepsilon\right]
=\widehat{\e}_0\left[\frac{g(Z)}{u(Z_1)};Z_1>\varepsilon\right]. 
\end{align}
Since $g$ is bounded,
\begin{align*}
\left|\widehat{\e}_x\left[\frac{g(Z)}{u(Z_1)};Z_1\le\varepsilon\right]\right|
&\le C_g \widehat{\e}_x\left[\frac{1}{u(Z_1)};Z_1\le\varepsilon\right]\\
&\le C_g\frac{\pr_x(T_0^{(Z)}>1)}{u(x)}
\pr_x(Z_1\le\varepsilon|T_0^{(Z)}>1).
\end{align*}
Using \eqref{eq:tail_small_x} and \eqref{eq:cond.lim.Z}, we conclude that 
\begin{align}
\label{eq:step.2}
\limsup_{x\to0}\left|\widehat{\e}_x\left[\frac{g(Z)}{u(Z_1)};Z_1\le\varepsilon\right]\right|
\le \frac{C_g}{(1-c)B(c,1-c)}\varepsilon.
\end{align}
Finally, recalling that the density of $Z_1$ under $\widehat{\pr}_0$ is proportional to $\frac{z^{1-c}}{(1+z)^2}$, we get 
\begin{align}
\label{eq:step.3}
\nonumber
\left|\widehat{\e}_0\left[\frac{g(Z)}{u(Z_1)};Z_1\le\varepsilon\right]\right|
&\le C_g \widehat{\e}_0\left[\frac{1}{u(Z_1)};Z_1\le\varepsilon\right]\\
&= C_g\int_0^\varepsilon (1+z)^{-2}dz
\le C_g\varepsilon.
\end{align}
Combining \eqref{eq:step.1}---\eqref{eq:step.3} and letting
$\varepsilon\to0$, we conclude that 
$$
\lim_{x\to0}
\widehat{\e}_x\left[\frac{g(Z)}{u(Z_1)}\right]
=\widehat{\e}_0\left[\frac{g(Z)}{u(Z_1)}\right].
$$
Noting now that 
$$
\e_x[g(Z)|T_0^{(Z)}>1]=\frac{u(x)}{\pr_x(T_0^{(Z)}>1)}
\widehat{\e}_x\left[\frac{g(Z)}{u(Z_1)}\right]
$$
and taking into account \eqref{eq:tail_small_x}, we obtain 
$$
\lim_{x\to\infty}\e_x[g(Z)|T_0^{(Z)}>1]
=(1-c)B(c,1-c)
\widehat{\e}_0\left[\frac{g(Z)}{u(Z_1)}\right].
$$
This completes the proof of the theorem.
\section{Construction of harmonic functions}
\label{sec:harmonic}
\subsection{Harmonic function for the random exchange process and for the maximal autoregressive process}
In this paragraph we shall consider the equation
\begin{equation}
\label{eq:harm.f}
G(x):=\e_x[G(R_1);T^{(R)}_{x_0}>1]=\e_x[G(R_1);R_1>x_0],\quad x>x_0.
\end{equation}

Assume first that $x\in(x_0,x_0+1]$. In this case one has
$$
\{R_1>x_0\}=\{R_1=\eta_1>x_0\}.
$$
Therefore,
$$
G(x)=\e[G(\eta_1);\eta_1>x_0]\quad\text{for all }x\in(x_0,x_0+1]. 
$$

For all $x>x_0+1$ one has $\pr_x(T^{(R)}_{x_0}>1)=1$. This implies that \eqref{eq:harm.f} reduces to 
\begin{align}
\label{eq:harm.f.2}
\nonumber
G(x)&=\e_x[G(R_1)]\\
&=G(x-1)\pr(\eta_1\le x-1)+\e[G(\eta_1);\eta_1>x-1],
\quad x>x_0+1. 
\end{align}
If $x\in(x_0+1,x_0+2]$ then $x-1\in(x_0,x_0+1]$ and, consequently,
$G(x-1)=G(x_0+1)$ for all $x\in(x_0+1,x_0+2]$. From this observation
and from \eqref{eq:harm.f.2} we have
\begin{align}
\label{eq:x0+2}
\nonumber
&G(x)\\
\nonumber
&=G(x_0+1)\pr(\eta_1\le x-1)+\e[G(\eta_1);\eta_1>x-1]\\
\nonumber
&=G(x_0+1)\pr(\eta_1\le x-1)+\e[G(\eta_1);\eta_1\in(x-1,x_0+1]]+\e[G(\eta_1);\eta_1>x_0+1]\\
&=G(x_0+1)\pr(\eta_1\le x_0+1)+\e[G(\eta_1);\eta_1>x_0+1].
\end{align}
This equality implies that $G(x)=G(x_0+2)$ for all $x\in(x_0+1,x_0+2]$.
Note also that
\begin{align*}
G(x_0+1)&=\e[G(\eta_1);\eta_1>x_0]\\
&=G(x_0+1)\pr(\eta_1\in(x_0,x_0+1]))+\e[G(\eta_1);\eta_1>x_0+1].
\end{align*}
Combining this with \eqref{eq:x0+2}, we conclude that
$$
G(x_0+2)=G(x_0+1)\left(1+\pr(\eta_1\le x_0)\right).
$$

Fix now an integer $n$ and consider the case $x\in(x_0+n,x_0+n+1]$.
Assume that we have already shown that $G(y)=G(x_0+n)$ for all
$y\in(x_0+n-1,x_0+n]$. Then we have from \eqref{eq:harm.f.2}
\begin{align*}
G(x)&=G(x-1)\pr(\eta_1\le x-1)+\e[G(\eta_1);\eta_1>x-1]\\
&=G(x_0+n)\pr(\eta_1\le x_0+n)+\e[G(\eta_1);\eta_1>x_0+n].
\end{align*}
Therefore, $G(x)=G(x_0+n+1)$ for all $x\in(x_0+n,x_0+n+1]$. This means that this property is valid for all $n$.

One has also equalities
$$
G(x_0+n+1)=G(x_0+n)\pr(\eta_1\le x_0+n)+\e[G(\eta_1);\eta_1>x_0+n]
$$
and
\begin{align*}
G(x_0+n)&=G(x_0+n-1)\pr(\eta_1\le x_0+n-1)+\e[G(\eta_1);\eta_1>x_0+n-1]\\
&=G(x_0+n-1)\pr(\eta_1\le x_0+n-1)\\
&\quad+G(x_0+n)\pr(\eta_1\in(x_0+n-1,x_0+n])
+\e[G(\eta_1);\eta_1>x_0+n].
\end{align*}
Taking the difference we obtain
\begin{align*}
&G(x_0+n+1)-G(x_0+n)\\
&\quad=G(x_0+n)\pr(\eta_1\le x_0+n)-G(x_0+n-1)\pr(\eta_1\le x_0+n-1)\\
&\hspace{2cm}-G(x_0+n)\pr(\eta_1\in(x_0+n-1,x_0+n])\\
&\quad=\pr(\eta_1\le x_0+n-1)\left(G(x_0+n)-G(x_0+n-1)\right).
\end{align*}
Consequently,
$$
G(x_0+n+1)-G(x_0+n)=
G(x_0+1)\prod_{k=0}^{n-1}\pr(\eta_1\le x_0+k),\quad n\ge1.
$$
As a result we have 
\begin{equation}
\label{eq:harm.f.3}
G(x)=G(x_0+1)\left(1+\sum_{j=1}^{n}\prod_{k=0}^{j-1}\pr(\eta_1\le x_0+k)\right),\ x\in(x_0+n,x_0+n+1].
\end{equation}
Finally, in order to get a non-trivial solution we have to show that
the equation 
$$
G(x_0+1)=\e[G(\eta_1);\eta_1>x_0]
$$
is solvable. In view of \eqref{eq:harm.f.3}, the previous equation is equivalent to
$$
G(x_0+1)=G(x_0+1)\sum_{n=0}^\infty 
\left(1+\sum_{j=1}^{n}\prod_{k=0}^{j-1}\pr(\eta_1\le x_0+k)\right)
\pr(\eta_1\in(x_0+n,x_0+n+1]).
$$
Now we infer that \eqref{eq:harm.f} has a non-trivial solution if and only if
$$
1=\sum_{n=0}^\infty 
\left(1+\sum_{j=1}^{n}\prod_{k=0}^{j-1}\pr(\eta_1\le x_0+k)\right)
\pr(\eta_1\in(x_0+n,x_0+n+1]).
$$
Clearly,
\begin{align*}
&\sum_{n=0}^\infty 
\left(1+\sum_{j=1}^{n}\prod_{k=0}^{j-1}\pr(\eta_1\le x_0+k)\right)
\pr(\eta_1\in(x_0+n,x_0+n+1])\\
&=\pr(\eta_1>x_0)+\sum_{j=1}^\infty\prod_{k=0}^{j-1}\pr(\eta_1\le x_0+k)
\sum_{n=j}^\infty \pr(\eta_1\in(x_0+n,x_0+n+1])\\
&=\pr(\eta_1>x_0)+\sum_{j=1}^\infty
\left(1-\pr(\eta_1\le x_0+j)\right)\prod_{k=0}^{j-1}\pr(\eta_1\le x_0+k).
\end{align*}
Furthermore, for every $N\ge 1$,
\begin{align*}
&\sum_{j=1}^N
\left(1-\pr(\eta_1\le x_0+j)\right)\prod_{k=0}^{j-1}\pr(\eta_1\le x_0+k)\\
&\hspace{1cm}=\sum_{j=1}^N\prod_{k=0}^{j-1}\pr(\eta_1\le x_0+k)
-\sum_{j=1}^N\prod_{k=0}^{j}\pr(\eta_1\le x_0+k)\\
&\hspace{1cm}=\pr(\eta_1\le x_0)-\prod_{k=0}^{N}\pr(\eta_1\le x_0+k).
\end{align*}
This implies that 
\begin{align*}
&\sum_{n=0}^\infty 
\left(1+\sum_{j=1}^{n}\prod_{k=0}^{j-1}\pr(\eta_1\le x_0+k)\right)
\pr(\eta_1\in(x_0+n,x_0+n+1])\\
&\hspace{1cm}=1-\lim_{N\to\infty}\prod_{k=0}^{N}\pr(\eta_1\le x_0+k).
\end{align*}
Thus, there is a non trivial solution $G(x)$ if and only if 
$$
\lim_{N\to\infty}\prod_{k=0}^{N}\pr(\eta_1\le x_0+k)=0.
$$
Noting that this is equivalent to $\e\eta_1^+=\infty$, we finish the proof of the first part of Theorem~\ref{thm:Rn}.
We notice also that $\e\eta_1^+=\infty$ implies that $\{R_n\}$
is either null recurrent or transient.

If $\{R_n\}$ is recurrent and $\pr(\eta_1\le x_0)>0$ then, according to \eqref{eq:recurrence}, the function $G(x)$ grows unboundedly. Furthermore, if \eqref{eq:log.tail.2} holds with some positive $c$ then it follows from the Karamata representation theorem that there exists a slowly varying function $L$ such that
\begin{equation}
\label{eq:karamata}
\prod_{k=0}^{j-1}\pr(\eta_1\le x_0+k)\sim \frac{L(j)}{j^{c}}
\quad\text{as }j\to\infty.
\end{equation}
If we assume that $c\in(0,1)$ then $\{R_n\}$ is null recurrent and 
\begin{equation}
\label{eq:harm.f.4}
G(x)\sim \frac{1}{1-c}x^{1-c}L(x),\quad x\to\infty.
\end{equation}

If $c=1$ then one has to take into account the asymptotic behaviour of the difference $\pr(\eta_1>y)-1/y$. Assume, for example, that
$$
\pr(\eta_1>y)=\frac{1}{y}+\frac{\theta+o(1)}{y\log y}
$$
for some $\theta\in(0,1)$.
Then $\{R_n\}$ is null recurrent and there exists a slowly varying function $L_1$ such that $L(x)\sim (\log x)^{-\theta}L_1(\log x)$. This implies that
$$
G(x)\sim \frac{1}{1-\theta}(\log x)^{1-\theta}L_1(\log x)
\quad\text{if }\theta<1.
$$

We conclude this paragraph with the following remark on the transient case. If $\{R_n\}$ is transient then the function 
$x\mapsto\pr_x(T^{(R)}_{x_0}=\infty)$ is harmonic and its limit, as $x\to\infty$, is equal to one. Then, according to \eqref{eq:harm.f.3},
\begin{equation}
\label{eq:return.prob}
\pr_x(T^{(R)}_{x_0}=\infty)=
\frac{1+\sum_{j\in[1,x-x_0)}\prod_{k=0}^{j-1}\pr(\eta_1\le x_0+k)}
{1+\sum_{j=1}^{\infty}\prod_{k=0}^{j-1}\pr(\eta_1\le x_0+k)},
\quad x>x_0.
\end{equation}
If \eqref{eq:log.tail.2} holds with some positive $c>1$ then the chain is transient and, using \eqref{eq:karamata}, we obtain
$$
\pr_x(T^{(R)}_{x_0}<\infty)\sim \frac{1}{(c-1)}\frac{L(x)}{x^{c-1}},\quad
x\to\infty.
$$

\subsection{Harmonic function for the autoregressive process:
proof of Theorem~\ref{thm:tail_of_T}(i)}
\begin{lemma}
\label{lem:mean_drift}
Let $W$ be an increasing, regularly varying of index $r\in(0,1)$
function. We assume also that 
$W'(x)=O\left(\frac{W(x)}{x}\right)$. If \eqref{eq:log.tail.2} holds then, as $z\to\infty$,
\begin{align*}
&\e[W(\log_A (A^{z-1}+A^{\eta_1}))]\\
&\hspace{1cm}=W(z-1)\pr(\eta_1\le z-1)+\e[W(\eta_1);\eta_1>z-1]
+o\left(\frac{W(z)}{z^2}\right).
\end{align*}
\end{lemma}
\begin{proof}
We start by decomposing the expectation into two parts:
\begin{align*}
&\e[W(\log_A (A^{z-1}+A^{\eta_1}))]\\
&=\e[W(\log_A (A^{z-1}+A^{\eta_1}));\eta_1\le z-1]
+\e[W(\log_A (A^{z-1}+A^{\eta_1}));\eta_1>z-1]\\
&=\e[W(z-1+\log_A(1+A^{\eta_1-z+1}));\eta_1\le z-1]\\
&\hspace{2cm}
+\e[W(\eta_1+\log_A (1+A^{z-1-\eta_1}));\eta_1>z-1].
\end{align*}
By the mean value theorem,
\begin{align*}
&\e[W(z-1+\log_A (1+A^{\eta_1-z+1}));\eta_1\le z-1]\\
&=W(z-1)\pr(\eta_1\le z-1)+
\e[W'(z-1+\theta_1)\log_A (1+A^{\eta_1-z+1});\eta_1\le z-1],
\end{align*}
where $\theta_1=\theta_1(z,\eta_1)\in(0,\log_A 2)$. Using now the assumption $W'(x)=O\left(\frac{W(x)}{x}\right)$, we obtain 
\begin{align*}
&\e[W(z-1+\log_A (+A^{\eta_1-z+1}));\eta_1\le z-1]\\
&=W(z-1)\pr(\eta_1\le z-1)+
O\left(\frac{W(z)}{z}\right)
\e[\log_A (1+A^{\eta_1-z+1});\eta_1\le z-1].
\end{align*}
It is easy to see that 
$$
\log_A (1+A^{\eta_1-z+1})=O\left(\frac{1}{z^2}\right)
$$
if $\eta_1\le z-1-2\log_Az$. Furthermore, \eqref{eq:log.tail.2} implies that 
$$
\pr(z-1-2\log_Az<\eta_1\le z-1)=o\left(\frac{1}{z}\right).
$$
Combining these relations, we infer that 
$$
\e[\log_A (1+A^{\eta_1-z+1});\eta_1\le z-1]
=o\left(\frac{1}{z}\right).
$$
As a result we have 
\begin{align}
\label{eq:L1.1}
\nonumber
&\e[W(z-1+\log_A (+A^{\eta_1-z+1}));\eta_1\le z-1]\\
&\hspace{2cm}
=W(z-1)\pr(\eta_1\le z-1)+o\left(\frac{W(z)}{z^2}\right).
\end{align}
Using the mean value theorem and the assumption
$W'(x)=O\left(\frac{W(x)}{x}\right)$ once again, we get 
\begin{align*}
&\e[W(\eta_1+\log_A (1+A^{z-1-\eta_1}));\eta_1>z-1]\\
&=\e[W(\eta_1);\eta_1>z-1]+
O\left(\frac{W(z)}{z}\right)
\e[\log_A (1+A^{z-1-\eta_1});\eta_1>z-1].
\end{align*}
Similar to the first part of the proof,
$$
\e[\log_A (1+A^{z-1-\eta_1});\eta_1>z-1]
=o\left(\frac{1}{z}\right).
$$
This leads to the equality 
\begin{align*}
&\e[W(\eta_1+\log_A (1+A^{z-1-\eta_1}));\eta_1>z-1]\\
&\hspace{2cm}
=\e[W(\eta_1);\eta_1>z-1]+o\left(\frac{W(z)}{z^2}\right).\end{align*}
Combining this with \eqref{eq:L1.1}, we obtain the desired equality.
\end{proof}
For every $\varepsilon\ge0$ we define
$$
u_\varepsilon(x)=(1+\varepsilon)\int_0^x\pr(\eta_1>y)dy,\quad x\ge0
$$
and
$$
U_\varepsilon(x)=
\left\{ 
\begin{array}{ll}
0, &x\le 0\\
\int_0^x e^{-u_\varepsilon(y)}dy, &x>0.
\end{array}
\right.
$$
\begin{lemma}
\label{lem:L2}
For every $\varepsilon\in[0,\frac{1-c}{c})$ one has
$$
\e[U_\varepsilon(\log_A(A^{z-1}+A^{\eta_1}))]
=U_\varepsilon(z)
-\frac{\varepsilon}{1+\varepsilon}e^{-u_\varepsilon(z)}
+O\left(\frac{U_\varepsilon(z)}{z^2}\right).
$$
\end{lemma}
\begin{proof}
\eqref{eq:log.tail.2} yields
$$
u_\varepsilon(x)\sim (1+\varepsilon)c\log x
\quad\text{as }x\to\infty.
$$
Furthermore, $U_\varepsilon(x)$ is regularly varying of index $1-c(1+\varepsilon)$ and that 
$$
U'_\varepsilon(x)=e^{-u_\varepsilon(x)}
\sim (1-c(1+\varepsilon))\frac{U_\varepsilon(x)}{x}. 
$$
Therefore, we may apply Lemma~\ref{lem:mean_drift} to the function $U_\varepsilon$:
\begin{align*}
&\e[U_\varepsilon(\log_A(A^{z-1}+A^{\eta_1}))]\\
&\hspace{1cm}=U_\varepsilon(z-1)\pr(\eta_1\le z-1)
+\e[U_\varepsilon(\eta_1);\eta_1>z-1]
+o\left(\frac{U_\varepsilon(z)}{z^2}\right).
\end{align*}
Integrating by parts, we have 
\begin{align*}
\e[U_\varepsilon(\eta_1);\eta_1>z-1]
&=U_\varepsilon(z-1)\pr(\eta_1>z-1)
+\int_{z-1}^\infty e^{-u_\varepsilon(y)}\pr(\eta_1>y)dy\\
&=U_\varepsilon(z-1)\pr(\eta_1>z-1)
+\frac{1}{1+\varepsilon}
\int_{z-1}^\infty e^{-u_\varepsilon(y)}u_\varepsilon'(y)dy\\
&=U_\varepsilon(z-1)\pr(\eta_1>z-1)
+\frac{1}{1+\varepsilon}e^{-u_\varepsilon(z-1)}.
\end{align*}
Consequently,
$$
\e[U_\varepsilon(\log_A(A^{z-1}+A^{\eta_1}))]
=U_\varepsilon(z-1)
+\frac{1}{1+\varepsilon}e^{-u_\varepsilon(z-1)}
+o\left(\frac{U_\varepsilon(z)}{z^2}\right).
$$
It remains now to notice that, by the Taylor formula, 
\begin{align*}
U_\varepsilon(z)
&=U_\varepsilon(z-1)+U_\varepsilon'(z-1)
+\frac{1}{2}U_\varepsilon''(z-1+\theta)\\
&=U_\varepsilon(z-1)+e^{-u_\varepsilon(z-1)}
+O\left(\frac{U_\varepsilon(z)}{z^2}\right).
\end{align*}
\end{proof}
Applying Lemma~\ref{lem:L2} and noting that $U_\varepsilon(x)=o(U_0(x))$, we get  
\begin{align*}
&\e_x[U_0(\log_A X_1)+U_\varepsilon(\log_A X_1)]\\
&\hspace{1cm}=\e[U_0(\log_A(A^{x-1}+A^{\eta_1}))
+U_\varepsilon(\log_A(A^{x-1}+A^{\eta_1}))]\\
&\hspace{1cm}=U_0(\log_A x)+U_\varepsilon(\log_A x)
-\frac{\varepsilon}{1+\varepsilon}e^{-u_\varepsilon(\log_A x)}
+O\left(\frac{U_0(\log_A x)}{(\log_A x)^2}\right).
\end{align*}
We know that $e^{-u_\varepsilon(z)}$ is regularly varying of index $-(1+\varepsilon)c$ and that $\frac{U_0(z)}{z^2}$ is regularly varying of index $-c-1$. Thus, for every $\varepsilon<\frac{1-c}{c}$ there exists $x^*$ such that 
\begin{equation*}
 \e_x[U_0(\log_A X_1)+U_\varepsilon(\log_A X_1)]
\le U_0(\log_A x)+U_\varepsilon(\log_A x),\quad x\ge x^*.
\end{equation*}
This inequality implies that if $x_0\ge x^*$ then the sequence
$$
Z_n:=U_0(\log_A X_{n\wedge T^{(X)}_{x_0}})
+U_\varepsilon(\log_A X_{n\wedge T^{(X)}_{x_0}})
$$
is a supermartingale. We next notice that 
\begin{align*}
&Z_{n+1}{\rm 1}\{T^{(X)}_{x_0}>n+1\}
-Z_{n}{\rm 1}\{T^{(X)}_{x_0}>n\}\\
&\hspace{2cm}=(Z_{n+1}-Z_{n}){\rm 1}\{T^{(X)}_{x_0}>n\}
-Z_{n+1}{\rm 1}\{T^{(X)}_{x_0}=n+1\}\\
&\hspace{2cm}\le (Z_{n+1}-Z_{n}){\rm 1}\{T^{(X)}_{x_0}>n\}.
\end{align*}
This implies that $Z_n{\rm 1}\{T^{(X)}_{x_0}>n\}$ is also a supermartingale. Consequently, the function
$$
V_\varepsilon(x)
:=\lim_{n\to\infty}
\e_x[U_0(\log_A X_n)+U_\varepsilon(\log_A X_n);
T^{(X)}_{x_0}>n]
$$
is finite. Furthermore, 
$$
V_\varepsilon(x)\le U_0(\log_A x)+U_\varepsilon(\log_A x)
\le CU_0(\log_A x),
\quad x>x_0.
$$
We now recall that $U_\varepsilon(z)=o(U_0(z))$. Thus, for every $\delta>0$ there exists $B$ such that $U_\varepsilon(z)\le\delta U_0(z)$ for all $z\ge B$. Therefore,
\begin{align*}
&\e_x[U_\varepsilon(\log_A X_n);T^{(X)}_{x_0}>n]\\
&\hspace{1cm}=\e_x[U_\varepsilon(\log_A X_n);\log_A X_n\le B,T^{(X)}_{x_0}>n]\\
&\hspace{3cm}
+\e_x[U_\varepsilon(\log_A X_n);\log_A X_n>B,T^{(X)}_{x_0}>n]\\
&\hspace{1cm}\le U_\varepsilon(B)\pr_x(T^{(X)}_{x_0}>n)
+\delta\e_x[U_0(\log_A X_n);T^{(X)}_{x_0}>n].
\end{align*}
Recalling that $\pr_X(T_{x_0}^{(X)}>n)\to0$, we get 
$$
\limsup_{n\to\infty}
\e_x[U_\varepsilon(\log_A X_n);T^{(X)}_{x_0}>n]
\le \delta V_\varepsilon(x).
$$
Letting now $\delta\to0$ we conclude that
$$
\lim_{n\to\infty}
\e_x[U_\varepsilon(\log_A X_n);T^{(X)}_{x_0}>n]=0
$$
This means that $V_\varepsilon$ does not depend on $\varepsilon$.
Thus we may set 
$$
V(x):=\lim_{n\to\infty}\e_x[U_0(\log_A X_n);T^{(X)}_{x_0}>n].
$$
Since $U_0$ and the chain $\{X_n\}$ are increasing, we infer that the function $V(x)$ is increasing as well.

By the Markov property,
$$
\e_x[U_0(\log_AX_{n+1});T^{(X)}_{x_0}>n+1]
=\int_{x_0}^\infty\pr_x(X_1\in dy)\e_y[U_0(\log_AX_{n});T^{(X)}_{x_0}>n].
$$
It follows from the supermartinale property of
$U_0(\log_AX_{n})+U_\varepsilon(\log_AX_{n})$ that 
\begin{align*}
\e_y[U_0(\log_AX_{n});T^{(X)}_{x_0}>n]
&\le \e_y[U_0(\log_AX_{n})+U_\varepsilon(\log_AX_{n});T^{(X)}_{x_0}>n]\\
&\le U_0(\log_Ay)+U_\varepsilon(\log_Ay),\quad n\ge 1
\end{align*}
This allows one to apply the dominated convergence theorem and to conclude that
$$
V(x)=\e_x[V(X_1);T^{(X)}_{x_0}>1],\quad x>x_0.
$$
In other words, $V(x)$ is harmonic for $X_n$ killed at $T^{(X)}_{x_0}$.
It is also clear that 
$$
V(x)
\le U_0(\log_A x)+U_\varepsilon(\log_A x)
\le CU_0(\log_A x).
$$

To show that this function is strictly positive we notice that 
$$
\e[U_0(\log_A(A^{z-1}+A^{\eta_1}))]
\ge U_0(z-1)\pr(\eta_1\le z-1)+\e[U_0(\eta_1);\eta_1>z-1].
$$
Using now the integration by parts, we get 
$$
\e[U_0(\log_A(A^{z-1}+A^{\eta_1}))]
\ge U_0(z-1)+e^{-u_0(z-1)}\ge U_0(z). 
$$
In other words, the sequence $U_0(\log_A X_n)$ is a submartingale. Then, by the optional stopping theorem,
$$
\e_x[U_0(\log_A X_n);T^{(X)}_{x_0}>n]
\ge U_0(\log_A x)-\e_x[U_0(\log_A X_{T^{(X)}_{x_0}});T^{(X)}_{x_0}\le n].
$$
Letting here $n\to\infty$, we conclude that 
$$
V(x)\ge U_0(\log_A x)-\e[U_0(\log_A X_{T^{(X)}_{x_0}})]
\ge U_0(\log_A x)-U_0(\log_A x_0).
$$
Thus, $V(x)>0$ for every $x>x_0$. Furthermore, one has the relation 
\begin{equation}
\label{eq:VsimU}
V(x)\sim U_0(\log_A x)\quad\text{as }x\to\infty.
\end{equation}
Summarizing, for each $x_0\ge x_*$ we have constructed a strictly positive on $(x_0,\infty)$, increasing harmonic function $V(x)$ such that
$V(A^x)\sim U_0(x)$.

We now turn to the case $x_0\le x_*$. Let $V_*$ be the function corresponding to the stopping time $T^{(X)}_{x_*}$, i.e.
$$
V_*(x)=\e_x[V_*(X_1);X_1>x_*], x>x_*.
$$
Define 
\begin{equation}
\label{eq:V.def}
V(x)=V_*(x){\rm 1}\{x>x_*\}
+\sum_{j=0}^\infty\int_{x_0}^{x_*}\pr_x(X_j\in dz, T^{(X)}_{x_0}>j)g(z),
\end{equation}
where
$$
g(z):=\e_z[V_*(X_1);X_1>x_*].
$$
Then one has 
\begin{align*}
&\e_x[V(X_1);X_1>x_0]\\
&\hspace{2mm}=\e_x[V_*(X_1);X_1>x_*]+\int_{x_0}^\infty\pr_x(X_1\in dy)
\sum_{j=0}^\infty\int_{x_0}^{x_*}\pr_y(X_j\in dz, T^{(X)}_{x_0}>j)g(z)\\
&\hspace{2mm}=\e_x[V_*(X_1);X_1>x_*]+
\sum_{j=1}^\infty\int_{x_0}^{x_*}\pr_x(X_j\in dz, T^{(X)}_{x_0}>j)g(z).
\end{align*}
If $x>x_*$ then 
$$
\e_x[V_*(X_1);X_1>x_*]=
V_*(x)=V_*(x)+\int_{x_0}^{x_*}\pr_x(X_0\in dz, T^{(X)}_{x_0}>0)g(z).
$$
Moreover, for $x\in(x_0,x_*]$ we have 
$$
\e_x[V_*(X_1);X_1>x_*]=
\int_{x_0}^{x_*}\pr_x(X_0\in dz, T^{(X)}_{x_0}>0)g(z).
$$
As a result,
$$
\e_x[V(X_1);X_1>x_0]=V(x),\quad x>x_0,
$$
i.e. $V$ is harmonic. 

Since $V(x)$ is strictly positive on the half-line
$(x_0,\infty)$, we may perform the corresponding Doob 
$h$-transform via the transition probabilities:
\begin{equation}
\label{eq:Doob-trans}
\widehat{\pr}^{(V)}_x(X_1\in dy)
=\frac{V(y)}{V(x)}\pr_x(X_1\in dy),
\quad x,y>x_0.
\end{equation}
The chain $X_n$ becomes transient under this new measure.
To see this we consider the sequence $(V(X_n))^{-1/2}$. It is immediate from the definition of $\widehat{\pr}$ that 
\begin{align*}
\widehat{\e}^{(V)}_x\left[\frac{1}{\sqrt{V(X_1)}}\right]
&=\frac{1}{V(x)}
\e_x\left[\frac{V(X_1)}{\sqrt{V(X_1)}}{\rm 1}\{T^{(X)}_{x_0}>1\}\right]\\
&=\frac{1}{V(x)}
\e_x\left[\sqrt{V(X_1)}{\rm 1}\{T^{(X)}_{x_0}>1\}\right].
\end{align*}
Applying now the Jensen inequality, we obtain 
\begin{align*}
\widehat{\e}_x^{(V)}\left[\frac{1}{\sqrt{V(X_1)}}\right]
\le\frac{1}{V(x)}\sqrt{\e_x\left[V(X_1){\rm 1}\{T^{(X)}_{x_0}>1\}\right]}=\frac{1}{\sqrt{V(x)}}.
\end{align*}
In other words, the sequence $(V(X_n))^{-1/2}$ is a positive 
supermartingale. Due to the Doob convergence theorem, this sequence converges almost surely. Noticing that
$\widehat{\pr}_x(\limsup X_n=\infty)=1$ 
implies that this limit of $(V(X_n))^{-1/2}$ is zero. This means that 
$$
X_n\to\infty\quad \widehat{\pr}^{(V)}-\text{a.s.}
$$

We now show that \eqref{eq:VsimU} holds also in the case when the harmonic function is defined by \eqref{eq:V.def}. Since the function $g(z)$ is increasing,
\begin{align*}
&\sum_{j=0}^\infty\int_{x_0}^{x_*}\pr_x(X_j\in dz, T^{(X)}_{x_0}>j)g(z)\\
&\hspace{1cm}\le g(x_*)
\sum_{j=0}^\infty\int_{x_0}^{x_*}\pr_x(X_j\le x_*, T^{(X)}_{x_0}>j)\\
&\hspace{1cm}\le g(x_*)\frac{\sup_{(x_0,x_*}V(x)}{\inf_{(x_0,x_*}V(x)}
\sum_{j=0}^\infty\int_{x_0}^{x_*}\widehat{\pr}_x(X_j\le x_*, T^{(X)}_{x_0}>j)\\
&\hspace{1cm}=C(x_0,x_*)\sum_{j=0}^\infty\int_{x_0}^{x_*}\widehat{\pr}_x(X_j\le x_*, T^{(X)}_{x_0}>j).
\end{align*}
Due to the transience of $\{X_n\}$ under $\widehat{P}$,
$$
\sum_{j=0}^\infty\int_{x_0}^{x_*}
\widehat{\pr}_x(X_j\le x_*, T^{(X)}_{x_0}>j)<\infty.
$$
Therefore,
$$
V(A^x)\sim V_*(A^x)\sim U_0(x).
$$
Thus, the proof of Theorem~\ref{thm:tail_of_T}(i) is complete.
\section{Lower and upper bounds for tails of recurrence times.}
\label{sec:l.u.bounds}
We shall consider the chain $X_n$ only and prove bounds in Theorem~\ref{thm:tail_of_T}(ii). The proofs of corresponding estimates for chains $M_n$ and $R_n$ are simpler.
\subsection{A lower bound for the tail of $T_{x_0}^{(X)}$}
We first consider the case $x_0\ge x_*$. As we have seen in the previous section, $V(x)$ is increasing in this case.

Let $\widehat{\pr}$ denote the Doob $h$-transform of $\pr$, for its definition see \eqref{eq:Doob-trans}.

Define 
$$
\sigma_y:=\inf\{n\ge1:X_n\ge A^y\}.
$$
By the total probability formula, for $x<A^{2n}$ and $B>2$,
\begin{align*}
&\widehat{\pr}_x(X_{\sigma_{2n}}\le A^{Bn},\sigma_{2n}\le n)\\
&=\sum_{k=1}^n\int_{x_0}^{A^{2n}}\widehat{\pr}_x\left(\sigma_{2n}>k-1,X_{k-1}\in dz\right)
\widehat{\pr}_{z}(X_1\in(A^{2n},A^{Bn}])\\
&\ge\inf_{z<A^{2n}}
\frac{\widehat{\pr}_{z}(X_1\in(A^{2n},A^{Bn}])}
{\widehat{\pr}_{z}(X_1>A^{2n})}
\sum_{k=1}^n\int_{x_0}^{A^{2n}}\widehat{\pr}_x\left(\sigma_{2n}>k-1,X_{k-1}\in dz\right)
\widehat{\pr}_{z}(X_1>A^{2n})\\
&=\inf_{z<A^{2n}}
\frac{\widehat{\pr}_{z}(X_1\in(A^{2n},A^{Bn}])}
{\widehat{\pr}_{z}(X_1>A^{2n})}
\widehat{\pr}_x(\sigma_{2n}\le n).
\end{align*}
For every $r\ge2$, using the integration by parts, we get 
\begin{align*}
&\widehat{\pr}_{z}(X_1>A^{rn})\\
&\hspace{0.5cm}=\frac{1}{V(z)}\int_{A^{rn}}^\infty V(y)\pr_z(X_1\in dy)\\
&\hspace{0.5cm}=\frac{1+o(1)}{V(z)}\int_{A^{rn}}^\infty U_0(\log_A y)\pr_z(X_1\in dy)\\
&\hspace{0.5cm}=\frac{1+o(1)}{V(z)}\left(U_0(rn)\pr_z(X_1>A^{rn})
+\int_{A^{rn}}^\infty U_0'(\log_A y)(\log_A y)'\pr_z(X_1>y)dy\right).
\end{align*}
According to \eqref{eq:log.tail.2},
$$
\pr_z(X_1>y)=\pr(\eta_1>\log_A(y-az))\sim\frac{c}{\log_A y}
$$
uniformly in $z\le A^{2n}$, $y\ge A^{2n}$. Therefore,
\begin{align*}
\widehat{\pr}_{z}(X_1>A^{rn})
&=\frac{1+o(1)}{V(z)}\left(U_0(rn)\frac{c}{rn}
+\int_{rn}^\infty e^{-u_0(t)}\pr(\eta_1>t)dt\right)\\
&=\frac{1+o(1)}{V(z)}\left(U_0(rn)\frac{c}{rn}
+e^{-u_0(rn)}\right)
=\frac{1+o(1)}{V(z)} r^{-c}\frac{U_0(n)}{n}.
\end{align*}
This implies that 
$$
\inf_{z<A^{2n}}
\frac{\widehat{\pr}_{z}(X_1\in(A^{2n},A^{Bn}])}
{\widehat{\pr}_{z}(X_1>A^{2n})}
=1-\left(\frac{2}{B}\right)^c+o(1).
$$
Taking $B=2^{1+2/c}$, we conclude that 
$$
\widehat{\pr}_x(X_{\sigma_{2n}}\le A^{2^{1+2/c}n},\sigma_{2n}\le n)
\ge \frac{1}{2}\widehat{\pr}_x(\sigma_{2n}\le n),\quad x\le A^{2n}
$$
for all $n$ large enough. Using this bound, we obtain 
\begin{align*}
&\widehat{\pr}_x(X_n\le A^{Bn})\\
&\hspace{1cm}\ge\widehat{\pr}_x(\sigma_{2n}>n)+
\widehat{\pr}_x(X_n\le A^{Bn},X_{\sigma_{2n}}\le A^{Bn},\sigma_{2n}\le n)\\
&\hspace{1cm}\ge\widehat{\pr}_x(\sigma_{2n}>n)
+\frac{1}{2}\widehat{\pr}(\sigma_{2n}\le n)
\widehat{\pr}_x(X_n\le A^{Bn}|X_{\sigma_{2n}}\le A^{Bn},\sigma_{2n}\le n).
\end{align*}
By the strong Markov property,
$$
\widehat{\pr}_x(X_n\le A^{Bn}|X_{\sigma_{2n}}\le A^{Bn},\sigma_{2n}\le n)
\ge \inf_{z\in (A^{2n},A^{Bn}]}\widehat{\pr}_z(X_j\le X_{j-1}\text{ for all }j\le n).
$$
If $X_0=z\ge A^{2n}$ then $X_j\ge A^{2n-j}$ for every $j\ge1$.
If $A^n\ge x_0$ then $\pr_z(T^{(X)}_{x_0}>n)=1$ and, consequently,
$$
\widehat{\pr}_z(X_j\le X_{j-1}\text{ for all }j\le n)
\ge\frac{V(A^n)}{V(z)}\pr_z(X_j\le X_{j-1}\text{ for all }j\le n)
$$
For every $y$ we have 
$$
\pr_y(X_1\le y)=\pr(\eta_1\le \log_A y+\log_A(1-a)).
$$
Thus, by the Markov property,
\begin{align*}
&\pr_z(X_j\le X_{j-1}\text{ for all }j\le n)\\
&\hspace{1cm}\ge \pr_z(X_j\le X_{j-1}\text{ for all }j\le n-1)
\pr(\eta_1\le n+1+\log_A(1-a))\\
&\hspace{1cm}\ge\ldots\ge 
\prod_{j=0}^{n-1}\pr(\eta_1\le 2n-j+\log_A(1-a))\\
&\hspace{1cm}\ge\left(
\pr(\eta_1\le n+1+\log_A(1-a))\right)^n\sim e^{-c}.
\end{align*}
This implies that 
$$
\inf_{z\in (A^{2n},A^{Bn}]}\widehat{\pr}_z(X_j\le X_{j-1}\text{ for all }j\le n)\ge\frac{V(A^n)}{V(A^{Bn})}\frac{e^{-c}}{2}\ge C_0.
$$
Therefore,
$$
\widehat{\pr}_x(X_n\le A^{Bn})
\ge \widehat{\pr}_x(\sigma_{2n}>n)
+\frac{C_0}{2}\widehat{\pr}_x(\sigma_{2n}\le n)
\ge \frac{C_0}{2}
$$
for all sufficiently large $n$.
Recalling that $V$ is increasing, we obtain the bound
\begin{align*}
\pr_x(T^{(X)}_{x_0}>n)
&=V(x)\widehat{\e}^{(V)}_x\left[\frac{1}{V(X_n)}\right]\\
&\ge \frac{V(x)}{V(A^{Bn})}\widehat{\pr}^{(V)}_x(X_n\le A^{Bn})
\ge \frac{e^{-c}}{4}\frac{V(x)}{V(A^{Bn})},\quad x\le A^{2n}.
\end{align*}
Consequently, 
\begin{equation}
\label{eq:lb}
\pr_x(T^{(X)}_{x_0}>n)\ge C_1\frac{V(x\wedge A^n)}{V(A^n)}
\end{equation}
for all $x>x_0\ge x_*$.

Thus, it remains to consider the case $x_0\le x_*$. If $x>x_*+1$ then 
$\pr_x(T^{(X)}_{x_0}>n)\ge \pr_x(T^{(X)}_{x_*}>n)$. Applying \eqref{eq:lb} with $x_0=x_*$, we get 
$$
\pr_x(T^{(X)}_{x_0}>n)\ge C_1\frac{V_*(x\wedge A^n)}{V_*(A^n)}
\ge C_2 \frac{V(x\wedge A^n)}{V(A^n)}.
$$
If $x\le x_*+1$ then 
$$
\pr_x(T^{(X)}_{x_0}>n)\ge
\pr(\xi_1>x_*+1)\pr_{x_*+1}(T^{(X)}_{x_0}>n-1)\ge 
C_3 \frac{V(x\wedge A^n)}{V(A^n)}.
$$
\subsection{Upper bound for $\pr(T^{(X)}_{x_0}>n)$}
\begin{lemma}
\label{lem:V-monotone}
If $V(x)$ is increasing on $(x_0,\infty)$ then 
$$
\pr_x(T^{(X)}_{x_0}>n)\le C\frac{V(x)}{V(A^n)},\quad x>x_0,\ n\ge 1.
$$
\end{lemma}
\begin{proof}
Since $\mathbf P_x(X_n>y)$ is monotonically increasing in $x$,
$$
\mathbf P_x(X_n>y\mid T^{(X)}_{x_0}>n) \ge \mathbf P_x(X_n>y) \quad \text{for all }x,y>x_0.
$$
Consequently,
$$
\mathbf E_x[W(X_n)\mid T^{(X)}_{x_0}>n] \ge \mathbf E_x[W(X_n)] \ge \mathbf E_x[W(X_n){\rm 1}\{X_n>x_0\}]
$$
for every nonnegative increasing function $W$.
(To prove this, one approximates $W$ by functions of the form $\sum_kc_k{\rm 1}_{(y_k,\infty)}$.)
In particular, for $W=V$ one gets
$$
V(x)=\mathbf E_x[V(X_n)\mid T^{(X)}_{x_0}>n]
\ge \mathbf E_x[V(X_n){\rm 1}\{X_n>x_0\}]$$
and can conclude
\begin{equation}
\label{eq:upperbound}
\mathbf P_x(T^{(X)}_{x_0}>n) 
\le \frac{V(x)}{\mathbf E_x[V(X_n){\rm 1}\{X_n>x_0\}]} 
\le \frac{V(x)}{\mathbf E_0[V(X_n){\rm 1}\{X_n>x_0\}]} \,.
\end{equation}
As we already know, $\frac{\log_A X_n}n$ converges weakly to the distribution with density 
$$
\frac{cy^{c-1}}{(y+1)^{c+1}}\mathbf1_{\mathbb R^+}(y)\,.
$$
The asymptotic behaviour in \eqref{eq:upperbound} is obtained most conveniently if one assumes that $\frac{\log_A X_n}n$ converges almost everywhere to some $Z$ with this distribution. (On a suitable probability space, the sequence can always be constructed in such a way.)
Then, as $v(x):=V(\log_A x)$ varies regularly with index $1-c$,
$$\frac{V(X_n)}{V(A^n)}
=\frac{v(\log_A X_n)}{v(n)}
=\frac{v\left(\frac{\log_A X_n}nn\right)}{v(n)}
\sim\frac{v(Zn)}{v(n)} \to Z^{1-c} \,.$$
(More precisely, due to the monotonicity of $V$, one first gets for every fixed $N\in\mathbb N$
\begin{align*}
\limsup_{n\to\infty}\frac{v\left(\frac{\log_A X_n}nn\right)}{v(n)}
\le \limsup_{n\to\infty}\frac{v\left(\sup_{k:k\ge N}\frac{\log_A X_k}kn\right)}{v(n)}
=\left(\sup_{k:k\ge N}\frac{\log_AX_k}k\right)^{1-c}\,.
\end{align*}
$N\to\infty$ shows $\limsup_{n\to\infty}\frac{v\left(\log_A X_n\right)}{v(n)}\le Z^{1-c}$ and likewise one checks that the lower limit has at least this value.)

Now one can apply the Fatou lemma:
\begin{equation*}
\begin{split}
\liminf_{n\to\infty}
\frac{\mathbf E_0[V(X_n){\rm 1}\{X_n>x_0\}]}{V(A^n)}
\ge \mathbf E[Z^{1-c}]
=\int_0^\infty y^{1-c}\frac{cy^{c-1}}{(y+1)^{c+1}} = 1 \,,
\end{split}
\end{equation*}
so \eqref{eq:upperbound} yields the desired bound.
\end{proof}
\begin{remark}
We know from the construction of $V$ that this function is increasing for $x_0\ge x_*$. We now notice that, using Theorem~\ref{thm:tail_of_T}(iii),
one can infer that $V$ is increasing for all $x_0$. Indeed, by the monotonicity of the chain $\{X_n\}$, 
$$
\pr_x(T^{(X)}_{x_0}>n)\le \pr_y(T^{(X)}_{x_0}>n)
$$
for all $n$ and for all $x\le y$. Combining this with the asymptotic relation $\pr_x(T^{(X)}_{x_0}>n)\sim\varkappa(c)\frac{V(x)}{V(A^n)}$, we conclude that $V(x)\le V(y)$. Thus, the bound in Lemma~\ref{lem:V-monotone} holds for each $x_0$ and, consequently, the upper bound in Theorem~\ref{thm:tail_of_T}(ii) is valid.
\end{remark}

We next prove an alternative upper bound, which is valid without monotonicity assumption. 
\begin{lemma}
\label{lem:down}
Assume that there exist $x_1$ and a subexponential distribution $F$ such that
$$
\pr_x(T^{(X)}_{x_1}>n)\le C(x)\overline{F}(n),\quad n\ge0,\ x> x_1.
$$
If $x_0<x_1$ is such that $\pr(ax_1+\xi_1<x_0)>0$ then there exists 
$C(x_0,x)$ such that 
$$
\pr_x(T^{(X)}_{x_0}>n)\le C(x,x_0)\overline{F}(n),\quad n\ge0,\ x> x_0.
$$
\end{lemma}
\begin{proof}
The assumption $\pr(ax_1+\xi_1<x_0)>0$ implies that 
$$
p:=\pr_{x_1}(X_{T^{(X)}_{x_1}}\le x_0)>0.
$$
Then we can represent the law of $T^{(X)}_{x_1}$ as a mixture of two distributions:
\begin{align*}
&\pr_{x_1}(T^{(X)}_{x_1}\in B)\\
&\hspace{1cm}=p\pr_{x_1}(T^{(X)}_{x_1}\in B|X_{T^{(X)}_{x_1}}\le x_0)
+(1-p)\pr_{x_1}(T^{(X)}_{x_1}\in B|X_{T^{(X)}_{x_1}}> x_0)\\
&\hspace{1cm}=:p\pr(\theta\in B)+(1-p)\pr(\zeta\in B).
\end{align*}
Noting that $\{X_n\}$ may visit $(x_0,x_1]$ several times before
$T^{(X)}_{x_0}$ and using the monotonicity of the chain, we get 
$$
\pr_{x_1}(T^{(X)}_{x_0}>n)
\le p\sum_{k=0}^\infty (1-p)^k\pr(\zeta_1+\zeta_2+\ldots+\zeta_k+\theta>n),
$$
where $\{\zeta_k\}$ are independent copies of $\zeta$.
Under the assumptions of the lemma we have 
$$
\pr(\zeta>n)\le C_1\overline{F}(n)
\quad\text{and}\quad
\pr(\theta>n)\le C_2\overline{F}(n).
$$
Then, by Proposition 4 in \cite{AFK03},
$$
\pr_{x_1}(T^{(X)}_{x_0}>n)\le C\overline{F}(n).
$$
If the starting point $x$ is smaller than $x_1$ then
$$
\pr_{x}(T^{(X)}_{x_0}>n)\le
\pr_{x_1}(T^{(X)}_{x_0}>n)\le C\overline{F}(n).
$$
If the starting point $x$ is bigger than $x_1$ then 
$\pr_{x}(T^{(X)}_{x_0}>n)$ is bounded by the tail of the convolution of 
$\pr_{x}(T^{(X)}_{x_1}\in\cdot)$ and $\pr_{x_1}(T^{(X)}_{x_0}\in\cdot)$.
Since the tails of these two distributions are $O(\overline{F}(n))$, the tail of their convolution is also $O(\overline{F}(n))$. This completes the 
proof of the lemma.
\end{proof}
\begin{corollary}
If \eqref{eq:log.tail.2} holds then 
$$
\pr_x(T^{(X)}_{x_0}>n)\le C\frac{V(x)}{V(A^n)}
$$
for all $x>x_0$ and all $n\ge1$.
\end{corollary}
\begin{proof}
It suffices to consider the case $x_0<x_*$.\\
Since $V_*(A^n)$ is regularly varying then, in view of
Lemma~\ref{lem:V-monotone}, the conditions of Lemma~\ref{lem:down} are valid for $x_1=x_*$ and $\overline{F}(n)\sim CU_0(n)$. Combining now Lemmata \ref{lem:V-monotone} and \ref{lem:down}, we have, for $x>x_*$, 
$$
\pr_x(T^{(X)}_{x_0}>n)
\le\pr_x(T^{(X)}_{x_*}>n/2)+\pr_{x_*}(T^{(X)}_{x_0}>n/2)
\le \frac{C_1V_*(x)+C_2}{V_*(A^{n/2})}
$$
Recalling that $V_*(A^n)$ is regularly varying and that 
$V(x)\le V_*(x)+C$ in the case $x_0<x_*$, we have the desired estimate for $x>x_*$. In the case $x\le x_*$ it suffices to apply Lemma~\ref{lem:down}.
\end{proof}
\section{Proof of asymptotic relations}
\label{sec:asym}
In this section we shall prove asymptotic relations in Theorem~\ref{thm:tail_of_T}(iii). Exact asymptotics in Theorem~\ref{thm:Rn}
can be derived by exactly the same arguments, and we omit their proof.

We are going to apply Theorem 3.10 from Durrett~\cite{Durrett78}
to the sequence of Markov processes 
$$
v^{(n)}_t:=\frac{\log_A X_{[nt]}}{n},\quad t\ge0.
$$ 
Since this sequence converges weakly to the process $Z$, which is non-degenerate and $\pr_x(T_0^{(Z)}>t)$ is strictly positive for all $x,t>0$, we conclude that the conditions (i)-(iii) from \cite{Durrett78} are fulfilled. Moreover, we have already shown 
that $\pr_x\left(\cdot\ |T_0^{(Z)}>1\right)$ converges, as
$x\to0$, to a non-degenerate limit. Thus, it remains to check that 
\begin{itemize}
 \item $\pr_{A^{nx_n}}(T^{(X)}_{x_0}>nt_n)\to\pr_x(T_0^{(Z)}>t)$
 if $x_n\to x>0$ and $t_n\to t>0$;
 \item $\pr_{A^{nx_n}}(T^{(X)}_{x_0}>nt_n)\to0$
 whenever $x_n\to0$ and $t_n\to t>0$;
 \item the sequence $v^{(n)}$ is tight; and 
 \item $\lim_{h\to0}\liminf_{n\to\infty}
 \pr_{A^x}(v^{(n)}_t>h|T^{(X)}_{x_0}>n)=1$ for every $t>0$.
\end{itemize}
We start with the first condition. 
\begin{lemma}
\label{lem:fixed_x_t}
If $x_n\to x>0$ and $t_n\to t>0$ then 
$$
\pr_{A^{nx_n}}(T^{(X)}_{x_0}>nt_n)\to\pr_x(T_0^{(Z)}>t).
$$
\end{lemma}
\begin{proof}
Since $\pr_{y}(T^{(X)}_{x_0}>m)$ is increasing in $y$ and decreasing in $m$, it suffices to prove the lemma in the special case 
$x_n=x$ and $t_n=t$. We are going to apply Theorem 2.1 from \cite{Durrett78}. We set
$$
A_0:=\left\{f\in D[0,t]:\inf_{s\le t}f(s)>0\right\}
$$
and
$$
A_n:=\left\{f\in D[0,t]:\inf_{s\le t}f(s)>\frac{x_0}{n}\right\},
\quad n\ge1.
$$
Furthermore, for every $\varepsilon>0$ we define
$$
G_\varepsilon
:=\left\{f\in D[0,t]:\inf_{s\le t}f(s)>\varepsilon\right\}.
$$
Then we have 
\begin{equation}
\label{eq:prob_boundary}
\pr_x(Z\in\partial G_\varepsilon)=0
\quad\text{for all }\varepsilon,t>0.
\end{equation}
It is clear that 
$G_\varepsilon\subset A_n$ for all $n>x_0/\varepsilon$ and 
$G_{1/n}\uparrow A_0$.
Thus, in order to apply Theorem 2.1 from \cite{Durrett78} we have only to show that 
\begin{equation}
\label{eq:ls}
\limsup_{n\to\infty}\pr_{A^{xn}}(T^{(X)}_{x_0}>nt)
\le\pr_x(T_0^{(Z)}>t).
\end{equation}
Fix some $\varepsilon<x$ and $\delta<t$.
Then, using the monotonicity of the chain $\{X_n\}$, we get
$$
\pr_{A^{xn}}(T^{(X)}_{x_0}>nt)\le
\pr_{A^{xn}}\left(\inf_{s\le t-\delta} X_{[sn]}>A^{\varepsilon n}\right)
+\pr_{A^{\varepsilon n}}(T^{(X)}_{x_0}>\delta n).
$$
According to the upper bound in \eqref{eq:tail_bounds},
$$
\pr_{A^{\varepsilon n}}(T^{(X)}_{x_0}>\delta n)
\le C\frac{V(A^{\varepsilon n})}{V(A^{\delta n})}.
$$
Recalling that $V(A^{x})$ is regularly varying of index $1-c$,
we conclude
$$
\limsup_{n\to\infty}\pr_{A^{\varepsilon n}}(T^{(X)}_{x_0}>\delta n)
\le C\left(\frac{\varepsilon}{\delta}\right)^{1-c}.
$$
Furthermore, combining \eqref{eq:weak.limit.1} and \eqref{eq:prob_boundary}, we get
$$
\pr_{A^{xn}}\left(\inf_{s\le t-\delta} X_{[sn]}>A^{\varepsilon n}\right)\to\pr_x\left(\inf_{s\le t-\delta}Z_s>\varepsilon\right).
$$
Consequently,
\begin{align*}
\limsup_{n\to\infty}\pr_{A^{xn}}(T^{(X)}_{x_0}>nt)
&\le \pr_x\left(\inf_{s\le t-\delta}Z_s>\varepsilon\right)
+C\left(\frac{\varepsilon}{\delta}\right)^{1-c}\\
&\le\pr_x(T_0^{(Z)}>t-\delta)
+C\left(\frac{\varepsilon}{\delta}\right)^{1-c}.
\end{align*}
Letting here first $\varepsilon\to0$ and then $\delta\to0$, we arrive at \eqref{eq:ls}. Thus, the proof is complete.
\end{proof}
\begin{lemma}
\label{lem:zero_start}
If $t_n\to t>0$ and $x_n\to0$ then 
$$
\pr_{A^{x_nn}}(T^{(X)}_{x_0}>nt)\to0.
$$
\end{lemma}
This is a simple consequence of the upper bound in \eqref{eq:tail_bounds} and we omit its proof.
\begin{lemma}
\label{lem: repulsion}
For all $x>x_0$ and all $t>0$ one has
$$
\lim_{h\to0}\limsup_{n\to\infty}
\pr_x(v^{(n)}_t>h|T^{(X)}_{x_0}>n)=1.
$$
\end{lemma}
\begin{proof}
By the definition of $v^{(n)}$,
$$
\pr_x(v^{(n)}_t\le h|T^{(X)}_{x_0}>n)
=\frac{\pr_x(X_{[nt]}\le A^{hn},T^{(X)}_{x_0}>n)}{\pr_x(T^{(X)}_{x_0}>n)}.
$$
Set $s=\min\{1,t\}/2$. Then, by the monotonicity of $X_n$,
$$
\pr_x(X_{[nt]}\le A^{hn},T^{(X)}_{x_0}>n)
\le \pr_x(T^{(X)}_{x_0}>ns)\pr_0(X_{[n(t-s)]}\le A^{hn}).
$$
Therefore,
$$
\pr_x(v^{(n)}_t>h|T^{(X)}_{x_0}>n)
\le\frac{\pr_x(T^{(X)}_{x_0}>ns)}{\pr_x(T^{(X)}_{x_0}>n)}
\pr_0(X_{[n(t-s)]}\le A^{hn}).
$$
Taking into account \eqref{eq:weak.limit.1}, \eqref{eq:tail_bounds} and \eqref{eq:trans.prob}, we get 
$$
\limsup_{n\to\infty}\pr_x(v^{(n)}_t>h|T^{(X)}_{x_0}>n)
\le Cs^{c-1}\pr_0(Z_{t-s}\le h)
\le Cs^{c-1}\left(\frac{h}{h+t-s}\right)^c.
$$
This yields the desired relation.
\end{proof}
To show the tightness we shall use the following upper bound for
the conditional distribution of $X_n$.
\begin{lemma}
\label{lem:cond_dist}
There exists a constant $C$ such that 
$$
\pr_x(X_n\ge A^{y}|T^{(X)}_{x_0}>n)\le C\frac{n}{y},
\quad y\ge 2\log_A\left(x+\frac{1}{1-a}\right).
$$
\end{lemma}
\begin{proof}
If $\xi_k< A^{y/2}$ for all $k\le n$ then
\begin{align*}
X_n&=a^nx+a^{n-1}\xi_1+a^{n-2}\xi_2+\ldots+\xi_n\\
&\le x+A^{y/2}\sum_{j=0}^{n-1}a^j
\le x+A^{y/2}\frac{1}{1-a}
\le A^y
\end{align*}
for all $y\ge 2\log_A\left(x+\frac{1}{1-a}\right)$.
Therefore,
\begin{align*}
\pr_x(X_n\ge A^y,T^{(X)}_{x_0}>n)
&\le\sum_{k=1}^n\pr_x(\xi_k\ge A^{y/2},T^{(X)}_{x_0}>n)\\
&\le\sum_{k=1}^n\pr_x(\xi_k\ge A^{y/2},T^{(X)}_{x_0}>k-1)\\
&\le\pr(\xi_1\ge A^{y/2})\sum_{k=1}^{n}\pr_x(T^{(X)}_{x_0}>k-1) 
\end{align*}
Using the upper bound in \eqref{eq:tail_bounds} and recalling that $V(A^x)$ is regularly varying with index $1-c$, we conclude that 
$$
\sum_{k=1}^{n}\pr_x(T^{(X)}_{x_0}>k-1)
\le 1+C\sum_{j=1}^{n-1}\frac{V(x)}{V(A^j)}
\le C\frac{nV(x)}{V(A^n)}.
$$
Consequently,
$$
\pr_x(X_n\ge A^y,T^{(X)}_{x_0}>n)
\le C \frac{nV(x)}{V(A^n)}\pr(\eta_1\ge y/2).
$$
Combining this with \eqref{eq:log.tail.2} and with the lower bound in \eqref{eq:tail_bounds}, we obtain the desired estimate.
\end{proof}
\begin{lemma}
\label{lem:tightness}
The sequence $v^{(n)}$ is tight.
\end{lemma}
\begin{proof}
According to Theorem 3.6 in \cite{Durrett78}, it suffices show that
\begin{equation}
\label{eq:tight.1}
\lim_{K\to\infty}\limsup_{n\to\infty}
\pr_x(X_n>A^{nK}|T^{(X)}_{x_0}>n)=0
\end{equation}
and
\begin{equation}
\label{eq:tight.2}
\lim_{t\to0}\limsup_{n\to\infty}
\pr_x(X_{[nt]}>A^{nh}|T^{(X)}_{x_0}>n)=0,
\quad h>0.
\end{equation}
\eqref{eq:tight.1} is immediate from Lemma~\ref{lem:cond_dist}.
To show \eqref{eq:tight.2} we first notice that, for every $t<1$, 
$$
\pr_x(X_{[nt]}>A^{nh}|T^{(X)}_{x_0}>n)
\le \pr_x(X_{[nt]}>A^{nh}|T^{(X)}_{x_0}>nt)
\frac{\pr_x(T^{(X)}_{x_0}>nt)}{\pr_x(T^{(X)}_{x_0}>n)}.
$$
Applying Lemma~\ref{lem:cond_dist} to the first probability term on the right hand side, we get 
$$
\pr_x(X_{[nt]}>A^{nh}|T^{(X)}_{x_0}>n)
\le C\frac{t}{h}
\frac{\pr_x(T^{(X)}_{x_0}>nt)}{\pr_x(T^{(X)}_{x_0}>n)}.
$$
Using again \eqref{eq:tail_bounds}, we have 
$$
\limsup_{n\to\infty}\frac{\pr_x(T^{(X)}_{x_0}>nt)}{\pr_x(T^{(X)}_{x_0}>n)}
\le Ct^{c-1}.
$$
As a result we have the estimate
$$
\limsup_{n\to\infty}\pr_x(X_{[nt]}>A^{nh}|T^{(X)}_{x_0}>n)
\le C\frac{t^c}{h},
$$
which implies \eqref{eq:tight.2}.
\end{proof}

We have checked all the conditions in Theorem 3.10 in \cite{Durrett78}. Therefore, the sequence of distributions 
$\pr_x\left(v^{(n)}\in\cdot|T^{(X)}_{x_0}>n\right)$ on $D[0,1]$ converges weakly towards the distribution $Q$ introduced in Theorem~\ref{thm:Z}.

Therefore, it remains to prove \eqref{eq:tail_asymp}. 
Since $V$ is harmonic,
\begin{align}
\label{eq:tail.1a}
\nonumber
V(x)
&=\e_x[V(X_n);T^{(X)}_{x_0}>n]\\
&=\e_x[V(X_n);T^{(X)}_{x_0}>n, X_n\le A^{Kn}]
+\e_x[V(X_n);T^{(X)}_{x_0}>n, X_n> A^{Kn}]
\end{align}
for every $K>0.$

We know that $V(x)\le CU_0(\log_A x)$. Therefore, 
\begin{align*}
&\e_x[V(X_n);T^{(X)}_{x_0}>n, X_n> A^{Kn}]\\
&\le C\e_x[U_0(\log_A X_n);T^{(X)}_{x_0}>n, X_n> A^{Kn}]\\
&=CU_0(Kn)\pr_x(\log_A X_n\ge Kn;T^{(X)}_{x_0}>n)\\
&\hspace{2cm}
+C\int_{Kn}^\infty U_0'(y)\pr_x(\log_A X_n\ge y;T^{(X)}_{x_0}>n)dy.
\end{align*}
Combining Lemma~\ref{lem:cond_dist} and \eqref{eq:tail_bounds},
we have
$$
\pr_x(\log_A X_n\ge y;T^{(X)}_{x_0}>n)\le C\frac{nV(x)}{yV(A^n)}.
$$
Consequently,
\begin{align*}
&\e_x[V(X_n);T^{(X)}_{x_0}>n, X_n> A^{Kn}]\\
&\hspace{1cm}\le C\frac{V(x)}{V(A^n)}\left(\frac{U_0(Kn)}{K}
+n\int_{Kn}^\infty \frac{U_0'(y)}{y}dy\right)\\
&\hspace{1cm}\le C\frac{V(x)}{V(A^n)}\left(\frac{U_0(Kn)}{K}
+n\int_{Kn}^\infty \frac{U_0(y)}{y^2}dy\right)
\le C\frac{V(x)}{V(A^n)}\frac{U_0(Kn)}{K}.
\end{align*}
Recalling that $V(A^n)\sim U_0(n)$ and that $U_0$ is regularly varying, we finally get 
\begin{equation}
\label{eq:tail.2a}
\limsup_{n\to\infty}\e_x[V(X_n);T^{(X)}_{x_0}>n, X_n> A^{Kn}]
\le \frac{C}{K^c}V(x).
\end{equation}
For the first summand on the right hand side of \eqref{eq:tail.1} we have 
\begin{align*}
&\e_x[V(X_n);T^{(X)}_{x_0}>n, X_n\le A^{Kn}]\\
&\hspace{1cm}=\pr_x(T^{(X)}_{x_0}>n)
\e_x[V(X_n){\rm 1}\{X_n\le A^{Kn}\}|T^{(X)}_{x_0}>n]\\
&\hspace{1cm}=V(A^n)\pr_x(T^{(X)}_{x_0}>n)
\e_x\left[\frac{V(X_n)}{V(A^n)}{\rm 1}\{X_n\le A^{Kn}\}
\Big|T^{(X)}_{x_0}>n\right].
\end{align*}
It follows from the already proven conditional limit theorem and from \eqref{eq:cond.lim.Z} that 
$$
\lim_{n\to\infty}
\pr_x\left(\frac{\log_A X_n}{n}\le y\Big|T^{(X)}_{x_0}>n\right)
=\frac{y}{y+1},\quad y>0.
$$
Combining this with the regular variation property of $V$, we obtain
\begin{align*}
&\e_x\left[\frac{V(X_n)}{V(A^n)}{\rm 1}\{X_n\le A^{Kn}\}
\Big|T^{(X)}_{x_0}>n\right]\\
&\hspace{1cm}=(1+o(1))
\e_x\left[\left(\frac{\log_A X_n}{n}\right)^{1-c}
{\rm 1}\{X_n\le A^{Kn}\}\Big|T^{(X)}_{x_0}>n\right]\\
&\hspace{1cm}=(1+o(1))
\int_0^K \frac{y^{c-1}}{(1+y)^2}dy.
\end{align*}
Consequently,
\begin{align}
\label{eq:tail.3a}
\nonumber
&\e_x[V(X_n);T^{(X)}_{x_0}>n, X_n\le A^{Kn}]\\
&\hspace{1cm}=(1+o(1))V(A^n)\pr_x(T^{(X)}_{x_0}>n)
\int_0^K \frac{y^{c-1}}{(1+y)^2}dy.
\end{align}
Plugging \eqref{eq:tail.2a} and \eqref{eq:tail.3a} into \eqref{eq:tail.1a} and letting $K\to\infty$, we obtain 
$$
\pr_x(T^{(X)}_{x_0}>n)
\sim\left(\int_0^\infty \frac{y^{c-1}}{(1+y)^2}dy\right)^{-1}
\frac{V(x)}{V(A^n)}.
$$
Thus, \eqref{eq:tail_asymp} holds with 
$$
\varkappa(c)
=\left(\int_0^\infty \frac{y^{c-1}}{(1+y)^2}dy\right)^{-1}
=\frac{1}{(1-c)B(c,1-c)}.
$$

\section{Proof of Theorem~\ref{thm:subex_for_R}}
\label{sec:subex_for_R}
\subsection{Expectation of hitting times for the maximal autoregressive process}
Put $u(x)=\E_x[T^{(R)}_{x_0}]$ and observe that the Markov property implies 
that it should satisfy 
\begin{align}
    \nonumber
u(x)&=
\pr_x(T^{(R)}_{x_0}=1)+
\e_x[1+u(R_1);T^{(R)}_{x_0}>1]\\
\label{eq:mean.hitting}
&=\pr_x(R_1\le x_0)+
\e_x[1+u(R_1);R_1>x_0],\quad x>x_0.
\end{align}

Assume first that $x\in(x_0,x_0+1]$. In this case one has
$$
\{R_1>x_0\}=\{R_1=\eta_1>x_0\}.
$$
Therefore,
\begin{align*}
u(x)&=
\pr(\eta_1\le x_0)+
\e[1+u(\eta_1);\eta_1>x_0]\\
&=1+ 
\e[u(\eta_1);\eta_1>x_0]\quad\text{for all }x\in(x_0,x_0+1].
\end{align*}

For all $x>x_0+1$ one has $\pr_x(T_{x_0}^{(R)}=1)=0$. 
This implies that \eqref{eq:mean.hitting} reduces to 
\begin{align}
\nonumber
u(x)&=\e_x[1+u(R_1)]\\
\label{eq:mean.hitting.2}
&=1+u(x-1)\pr(\eta_1\le x-1)+\e[u(\eta_1);\eta_1>x-1],
\quad x>x_0+1. 
\end{align}
If $x\in(x_0+1,x_0+2]$ then $x-1\in(x_0,x_0+1]$ and, consequently,
$u(x-1)=u(x_0+1)$ for all $x\in(x_0+1,x_0+2]$. From this observation
and from \eqref{eq:mean.hitting.2} we have
\begin{align}
\label{eq:x0+2.hitting}
\nonumber
u(x)
&=1+u(x_0+1)\pr(\eta_1\le x-1)+\e[u(\eta_1);\eta_1>x-1]\\
\nonumber
&=1+u(x_0+1)\pr(\eta_1\le x-1)\\
\nonumber
&\hspace{1cm}+\e[u(\eta_1);\eta_1\in(x-1,x_0+1]]
+\e[u(\eta_1);\eta_1>x_0+1]\\
&=1+u(x_0+1)\pr(\eta_1\le x_0+1)+\e[u(\eta_1);\eta_1>x_0+1].
\end{align}
This equality implies that $u(x)=u(x_0+2)$ for all $x\in(x_0+1,x_0+2]$.
Note also that 
\begin{align*}
u(x_0+1)&=1+\e[u(\eta_1);\eta_1>x_0]\\
&=1+u(x_0+1)\pr(\eta_1\in(x_0,x_0+1]))+\e[u(\eta_1);\eta_1>x_0+1].
\end{align*}
Combining this with \eqref{eq:x0+2.hitting}, we conclude that
$$
u(x_0+2)=u(x_0+1)\left(1+\pr(\eta_1\le x_0)\right).
$$

Fix now an integer $n$ and consider the case $x\in(x_0+n,x_0+n+1]$.
Assume that we have already shown that $u(y)=u(x_0+n)$ for all
$y\in(x_0+n-1,x_0+n]$. Then we have from \eqref{eq:mean.hitting.2}
\begin{align*}
u(x)&=1+u(x_0+n)\pr(\eta_1\le x-1)+\e[u(\eta_1);\eta_1>x-1]\\
&=1+u(x_0+n)\pr(\eta_1\le x_0+n)+\e[u(\eta_1);\eta_1>x_0+n].
\end{align*}
Therefore, $u(x)=u(x_0+n+1)$ for all $x\in(x_0+n,x_0+n+1]$. 
This means that this property is valid for all $n$.

One has also equalities
$$
u(x_0+n+1)=1+u(x_0+n)\pr(\eta_1\le x_0+n)+\e[u(\eta_1);\eta_1>x_0+n]
$$
and
\begin{align*}
u(x_0+n)&=1+u(x_0+n-1)\pr(\eta_1\le x_0+n-1)+\e[u(\eta_1);\eta_1>x_0+n-1]\\
&=1+u(x_0+n-1)\pr(\eta_1\le x_0+n-1)\\
&\quad+u(x_0+n)\pr(\eta_1\in(x_0+n-1,x_0+n])
+\e[u(\eta_1);\eta_1>x_0+n].
\end{align*}
Taking the difference we obtain
\begin{align*}
&u(x_0+n+1)-u(x_0+n)\\
&\quad=u(x_0+n)\pr(\eta_1\le x_0+n)-u(x_0+n-1)\pr(\eta_1\le x_0+n-1)\\
&\hspace{2cm}-u(x_0+n)\pr(\eta_1\in(x_0+n-1,x_0+n])\\
&\quad=\pr(\eta_1\le x_0+n-1)\left(u(x_0+n)-u(x_0+n-1)\right).
\end{align*}
Consequently,
$$
u(x_0+n+1)-u(x_0+n)=
u(x_0+1)\prod_{k=0}^{n-1}\pr(\eta_1\le x_0+k),\quad n\ge 1.
$$
As a result, we have the following expression for the expectation of the 
 hitting time 
\begin{equation}
\label{eq:mean.hitting.3}
u(x)=
u(x_0+1)\left(1+\sum_{j=1}^{n}\prod_{k=0}^{j-1}\pr(\eta_1\le x_0+k)\right),\ x\in(x_0+n,x_0+n+1].
\end{equation}
Finally, in order to get a finite solution we have to show that
the equation 
$$
u(x_0+1)=1+\e[u(\eta_1);\eta_1>x_0]
$$
is solvable. In view of \eqref{eq:mean.hitting.3}, 
the previous equation is equivalent to
$$
u(x_0+1)=1+u(x_0+1)\sum_{n=0}^\infty 
\left(1+\sum_{j=1}^{n}\prod_{k=0}^{j-1}\pr(\eta_1\le x_0+k)\right)
\pr(\eta_1\in(x_0+n,x_0+n+1]).
$$
Now we conclude that \eqref{eq:mean.hitting} 
has a finite solution if and only if
$$
\sum_{n=0}^\infty 
\left(1+\sum_{j=1}^{n}\prod_{k=0}^{j-1}\pr(\eta_1\le x_0+k)\right)
\pr(\eta_1\in(x_0+n,x_0+n+1])<1.
$$
Clearly,
\begin{align*}
&\sum_{n=0}^\infty 
\left(1+\sum_{j=1}^{n}\prod_{k=0}^{j-1}\pr(\eta_1\le x_0+k)\right)
\pr(\eta_1\in(x_0+n,x_0+n+1])\\
&=\pr(\eta_1>x_0)+\sum_{j=1}^\infty\prod_{k=0}^{j-1}\pr(\eta_1\le x_0+k)
\sum_{n=j}^\infty \pr(\eta_1\in(x_0+n,x_0+n+1])\\
&=\pr(\eta_1>x_0)+\sum_{j=1}^\infty
\left(1-\pr(\eta_1\le x_0+j)\right)\prod_{k=0}^{j-1}\pr(\eta_1\le x_0+k).
\end{align*}
Furthermore, for every $N\ge 1$,
\begin{align*}
&\sum_{j=1}^N
\left(1-\pr(\eta_1\le x_0+j)\right)\prod_{k=0}^{j-1}\pr(\eta_1\le x_0+k)\\
&\hspace{1cm}=\sum_{j=1}^N\prod_{k=0}^{j-1}\pr(\eta_1\le x_0+k)
-\sum_{j=1}^N\prod_{k=0}^{j}\pr(\eta_1\le x_0+k)\\
&\hspace{1cm}=\pr(\eta_1\le x_0)-\prod_{k=0}^{N}\pr(\eta_1\le x_0+k).
\end{align*}
This implies that 
\begin{align*}
&\sum_{n=0}^\infty 
\left(1+\sum_{j=1}^{n}\prod_{k=0}^{j-1}\pr(\eta_1\le x_0+k)\right)
\pr(\eta_1\in(x_0+n,x_0+n+1])\\
&\hspace{1cm}=1-\lim_{N\to\infty}\prod_{k=0}^{N}\pr(\eta_1\le x_0+k).
\end{align*}
Thus, there is a finite solution $u(x)$ if and only if 
$$
\lim_{N\to\infty}\prod_{k=0}^{N}\pr(\eta_1\le x_0+k)>0.
$$
Note that this is equivalent to $\e\eta_1^+<\infty$. 
Then, 
\begin{align*}
u(x_0+1)= \frac{1}{\prod_{k=0}^{\infty}\pr(\eta_1\le x_0+k)}. 
\end{align*}

\subsection{Recursion for tails of exit times}
We will consider now $\pr_x(T^{(R)}_{x_0}>n)$. 
Define
$$
v(n,k) = \pr_{x_0+k+1}(T^{(R)}_{x_0}>n),\quad n,k\ge 0
$$
and 
$$
v_n=v(n,0).
$$
Then the following result holds.  
\begin{proposition}\label{prop:tail.recursion}
Assume that $0<\pr(\eta_1<x_0)<1$. 
Then, for integer $n,k\ge 0$, 
\begin{equation}\label{eq:tail.0}
\pr_x(T^{(R)}_{x_0}>n) = v(n,k), \quad x\in (x_0+k,x_0+k+1]. 
\end{equation}
For $n\le k$ we have $v(n,k)=1$ and for $n> k$ the following 
recursive equality holds 
\begin{equation}\label{eq:tail.9}
    v(n,k) 
    = v(n,0)+\sum_{m=1}^{k} v(n-m,0)\prod_{j=0}^{m-1} \pr(\eta_1\le x_0+j).
\end{equation}
Furthermore,
\begin{align}\label{eq:tail.8}
    v_n&=\pr(\eta_1>x_0+n-1)+ 
    v_{n-1} \pr(\eta_1\in (x_0,x_0+n-1]) \\
    \nonumber 
    &+\sum_{m=1}^{n-2} v_{n-m-1}\pr(\eta_1\in (x_0+m,x_0+n-1])
    \prod_{j=0}^{m-1} 
        \pr(\eta_1\le x_0+j)
\end{align}
and hence~\eqref{eq:tail.9} and~\eqref{eq:tail.8} allow us to find 
$\pr_x(T^{(R)}_{x_0}>n)$ recursively. 
\end{proposition}    
\begin{proof}
It is clear that for $n\le k$ it holds $\pr_x(T^{(R)}_{x_0}>n)=v(n,k)=1$. 
Hence, in the rest of the proof we will  assume  that $n>k$. 

Let $x\in (x_0,x_0+1]$. Then,  for $n>0$ we have,
\begin{align}\label{eq:tail.1}
\nonumber
\pr_x(T^{(R)}_{x_0}>n)
&=\int_{x_0}^\infty \pr_x(R_1\in dy)\pr_y(T^{(R)}_{x_0}>n-1)\\
&= \int_{x_0}^\infty \pr(\eta_1\in dy)\pr_y(T^{(R)}_{x_0}>n-1). 
\end{align}    
Clearly this probability is the same for each $x\in (x_0,x_0+1]$ and 
hence~\eqref{eq:tail.0}  holds for $k=0$. 

Next consider $x\in (x_0+1,x_0+2]$. 
For every $n>1$ we have 
\begin{align}
    \nonumber
    &\pr_x(T^{(R)}_{x_0}>n)\\
    \nonumber
    &=\pr_{x-1}(T^{(R)}_{x_0}>n-1)\pr(\eta_1 \le x-1 )
    +\int_{x-1}^\infty\pr(\eta_1\in dy)\pr_y(T^{(R)}_{x_0}>n-1)\\
    \nonumber
    &=\pr_{x_0+1}(T^{(R)}_{x_0}>n-1)\pr(\eta_1 \le x-1 )
    +\int_{x-1}^{x_0+1}\pr(\eta_1\in dy)\pr_{x_0+1} (T^{(R)}_{x_0}>n-1 )
    \\
    \nonumber 
    &\hspace{1cm}    
    +\int_{x_0+1}^\infty\pr(\eta_1\in dy)\pr_y(T^{(R)}_{x_0}>n-1)\\ 
   \label{eq:tail.3}  
    &= v(n-1,0)\pr(\eta_1\le x_0+1)+\int_{x_0+1}^\infty 
    \pr(\eta_1\in dy)\pr_y(T^{(R)}_{x_0}>n-1). 
\end{align}
This expression is constant for $x\in (x_0+1,x_0+2]$
 and hence~\eqref{eq:tail.0} holds for $k=1$.  
Note also that it follows from~\eqref{eq:tail.1} that 
\[
v(n,0)=v(n-1,0)\pr(\eta_1\in (x_0,x_0+1]) 
+\int_{x_0+1}^\infty 
    \pr(\eta_1\in dy)\pr_y(T^{(R)}_{x_0}>n-1).
\]
Subtracting this expression from~\eqref{eq:tail.3} we obtain 
\begin{equation}\label{eq:tail.4}
    v(n,1)-v(n,0) = v(n-1,0)\pr(\eta_1\le x_0). 
\end{equation}    

We will now prove by induction that 
for $x\in (x_0+k,x_0+k]$ 
the tail $\pr_x(T^{(R)}_{x_0}>n)$ is constant and  will simultaneously show that for $k\ge 2$ that 
\begin{equation}\label{eq:tail.4a}
    v(n,k)-v(n,k-1) = (v(n-1,k-1)-v(n-1,k-2))\pr(\eta_1\le x_0+k-1). 
\end{equation}    

First consider the base of induction $k=2$. 
In this case, for $n>2$ 
and for $x\in (x_0+2,x_0+3]$, we have  
\begin{align*}
    &\pr_x(T^{(R)}_{x_0}>n)\\
    &=\pr_{x-1}(T^{(R)}_{x_0}>n-1)\pr(\eta_1\le x-1 ) +\int_{x-1}^\infty 
    \pr(\eta_1\in dy)\pr_y(T^{(R)}_{x_0}>n-1)\\
    &=\pr_{x_0+2}(T^{(R)}_{x_0}>n-1)+\int_{x-1}^{x_0+2}\pr(\eta_1\in dy)\pr_{x_0+2} (T^{(R)}_{x_0}>n-1)\\
    &\hspace{1cm}+\int_{x_0+2}^\infty 
    \pr(\eta_1\in dy)\pr_y(T^{(R)}_{x_0}>n-1)\\    
    &= v(n-1,1)\pr(\eta_1\le x_0+2)+\int_{x_0+2}^\infty 
    \pr(\eta_1\in dy)\pr_y(T^{(R)}_{x_0}>n-1). 
\end{align*}
This expression clearly does not depend on $x$. Thus, 
\begin{equation}
    \label{eq:tail.5}  
v(n,2)= v(n-1,1)\pr(\eta_1\le x_0+2)+\int_{x_0+2}^\infty 
\pr(\eta_1\in dy)\pr_y(T^{(R)}_{x_0}>n-1). 
\end{equation}
It also follows from~\eqref{eq:tail.3} that 
\begin{align*}
v(n,1)&=v(n-1,0)\pr(\eta_1\le x_0+1)    
+v(n-1,1)\pr(\eta_1\in (x_0+1,x_0+2])\\
&\hspace{1cm}+\int_{x_0+2}^\infty 
\pr(\eta_1\in dy)\pr_y(T^{(R)}_{x_0}>n-1).
\end{align*}
Subtracting this equation from~\eqref{eq:tail.5} we obtain 
\[
v(n,2)-v(n,1)=   
(v(n-1,1)-v(n-1,0)) 
\pr(\eta_1\le x_0+1).      
\]
This is exactly~\eqref{eq:tail.4a} with $k=2$. 
Thus, the base case is true. 

We will now prove the induction step. 
Consider $x\in (x_0+k,x_0+k+1]$. 
For $n>k$ we obtain, using the induction hypothesis, 
\begin{align*}
    &\pr_x(T^{(R)}_{x_0}>n)\\
    &=\pr_{x-1}(T^{(R)}_{x_0}>n-1)\pr(\eta_1\le x-1 )
    +\int_{x-1}^\infty \pr(\eta_1\in dy)\pr_y(T^{(R)}_{x_0}>n-1)\\
    &=\pr_{x_0+k}(T^{(R)}_{x_0}>n-1)+\int_{x-1}^{x_0+k}\pr(\eta_1\in dy)\pr_{x_0+k} (T^{(R)}_{x_0}>n-1)\\
    &\hspace{1cm}+\int_{x_0+k}^\infty 
    \pr(\eta_1\in dy)\pr_y(T^{(R)}_{x_0}>n-1)\\    
    &= v(n-1,k-1)\pr(\eta_1\le x_0+k)+\int_{x_0+k}^\infty 
    \pr(\eta_1\in dy)\pr_y(T^{(R)}_{x_0}>n-1). 
\end{align*}
This expression clearly does not depend on $x$ and hence~\eqref{eq:tail.0} holds. Thus, 
\begin{equation}
    \label{eq:tail.6}  
v(n,k)= v(n-1,k-1)\pr(\eta_1\le x_0+k)+\int_{x_0+k}^\infty 
\pr(\eta_1\in dy)\pr_y(T^{(R)}_{x_0}>n-1). 
\end{equation}
The same expression is true for $k-1$ by the induction hypothesis. 
Hence,
\begin{align*}
    v(n,k-1)&=v(n-1,k-2)\pr(\eta_1\le x_0+k-1)\\    
    &\hspace{1cm}+v(n-1,k-1)\pr(\eta_1\in (x_0+k-1,x_0+k])\\ 
    &\hspace{1cm}+\int_{x_0+k}^\infty 
    \pr(\eta_1\in dy)\pr_y(T^{(R)}_{x_0}>n-1).
\end{align*}
Subtracting this expression from~\eqref{eq:tail.6} we obtain~\eqref{eq:tail.4a}.  

Now it follows from~\eqref{eq:tail.4} and~\eqref{eq:tail.4a}  that 
\begin{equation}\label{eq:tail.7}
v(n,k)-v(n,k-1) = v(n-k,0) \prod_{j=0}^{k-1} \pr(\eta_1\le x_0+j).
\end{equation}
for $n>k$. Then the standard telescoping argument gives~\eqref{eq:tail.9}. 
Plugging~\eqref{eq:tail.7} into~\eqref{eq:tail.1} we obtain 
\begin{align*}
    v_n &=\int_{x_0+1}^\infty  \pr(\eta_1\in dy) \pr_y(T^{(R)}_{x_0}>n-1)
    +v_{n-1}\pr(\eta_1\in(x_0,x_0+1]) 
    \\
    &=\sum_{l=1}^{n-2} \pr(\eta_1\in (x_0+l,x_0+l+1])v(n-1,l)+\pr(\eta_1>x_0+n-1)\\
    &\hspace{1cm}+v_{n-1}\pr(\eta_1\in(x_0,x_0+1])  
    \\
    &=\sum_{l=1}^{n-2} \pr(\eta_1\in (x_0+l,x_0+l+1]) 
    \left(
        v_{n-1}+\sum_{m=1}^{l} v_{n-m}\prod_{j=0}^{m-1} 
        \pr(\eta_1\le x_0+j)
    \right)\\
    &\hspace{1cm} +\pr(\eta_1>x_0+n-1)+v_{n-1}\pr(\eta_1\in(x_0,x_0+1]). 
\end{align*}
Swapping the order of summation we obtain~\eqref{eq:tail.8}. 

\end{proof} 

\subsection{Heavy tails}
To analyse the heavy-tailed 
case we need first the following 
definition. 
We say that a non-negative sequence $(a_n)_{n\ge 0}$ is 
subexponential if 
\begin{align*}
    \lim_{n\to\infty}\frac{a_n}{a_{n-1}}=1, \quad 
    a_\infty:=
    \sum_{n=0}^\infty a_n<\infty\\
    \sum_{k=0}^n a_k a_{n-k}
    \sim 2 a_\infty a_n,\quad n\to\infty. 
\end{align*}
We start by deriving an upper bound for $v_n$. 
\begin{lemma}\label{lem:tail.upper}
Assume that $F\in \mathcal S^*$, where $F(x)=\pr(\eta_1\le x)$. 
Then there exists a constant $C$ such that 
\[
v_n\le C\pr(\eta_1>n),\quad n\ge 0.       
\]
\end{lemma}   
\begin{proof}
Note that it follows from~\eqref{eq:tail.8} that $v_n\le w_n$, where the sequence $\{w_n\}$ is given by 
    $w_1=\pr(\eta_1>x_0)$  and for $n\ge 2$, 
    \begin{align*}
        w_n&=\pr(\eta_1>x_0+n-1)+ 
        w_{n-1} \pr(\eta_1> x_0) \\ 
        \nonumber 
        &+\sum_{m=1}^{n-2} w_{n-m-1}\pr(\eta_1> x_0+m)
        \prod_{j=0}^{m-1} 
            \pr(\eta_1\le x_0+j). 
    \end{align*}    
Set $d_0=\pr(\eta_1> x_0)$ and 
\[
d_m :=  \pr(\eta_1> x_0+m)
\prod_{j=0}^{m-1} 
    \pr(\eta_1\le x_0+j),\quad m\ge 1.  
\]
Set also $c_n=\pr(\eta_1>x_0+n-1)$, $n\ge1$
Then we have $w_1=c_1$ and
\begin{align}
\label{defn:wn}
w_n=c_n+\sum_{m=0}^{n-2}w_{n-m-1}d_m,\quad n\ge2.
\end{align}
Using~\eqref{defn:wn}, we obtain the following equality for generating functions:
\begin{align*}
\sum_{n=1}^\infty w_ns^n
&=\sum_{n=1}^\infty c_ns^n
+\sum_{n=2}^\infty s^n\sum_{m=0}^{n-2}w_{n-m-1}d_m\\
&=\sum_{n=1}^\infty c_ns^n
+s\sum_{m=0}^\infty d_ms^m\sum_{n=m+2}^\infty w_{n-m-1}s^{n-m-1}\\
&=\sum_{n=1}^\infty c_ns^n
+s\sum_{m=0}^\infty d_ms^m\left(\sum_{n=1}^\infty w_ns^n\right).
\end{align*}    
Set, for brevity, 
$$
\widehat{w}(s)=\sum_{n=1}^\infty w_ns^n,\ 
\widehat{c}(s)=\sum_{n=1}^\infty c_ns^n,\text{ and }
\widehat{d}(s)=\sum_{n=0}^\infty d_ns^n. 
$$
Then we have 
$$
\widehat{w}(s)=\widehat{c}(s)+s\widehat{d}(s)\widehat{w}(s).
$$
Solving this equality we obtain 
\[
\widehat w(s) = \frac{\widehat c(s)}{1-s \widehat d(s)}. 
\]
Noting that
\begin{align*}
d_m
&=(1-\pr(\eta_1\le x_0+m))\prod_{j=0}^{m-1}\pr(\eta_1\le x_0+j))\\
&=\prod_{j=0}^{m-1}\pr(\eta_1\le x_0+j))-\prod_{j=0}^{m}\pr(\eta_1\le x_0+j)),
\end{align*}
we get 
\begin{align*}
\sum_{m=0}^\infty d_m
=1-\prod_{j=0}^{\infty}\pr(\eta_1\le x_0+j))<1,
\end{align*}
where the last inequality follows from the assumption $\e\eta<\infty$.
Also it is clear that 
\[
\frac{d_{n+1}}{d_n}=
\pr(\eta_1\le n+x_0)\to 1, 
\quad n\to\infty
\]
and 
\[
d_n\sim \pr(\eta_1>n)
\prod_{j=0}^\infty \pr(\eta_1\le x_0+j).
\]
Since $F\in \mathcal S^*$ 
we can see that $(d_n)_{n\ge 0}$ is a subexponential sequence.

Then, it follows from the results in  the theory of locally subexponential distributions (see Corollary~2 and Proposition~4 in~\cite{AFK03})  that 
$\frac{1}{1-s \widehat d(s)}$ is a  generating function of subexponential 
sequence behaving like $C_2\pr(\eta_1>n)$. 
The same  statement holds for  $\widehat c(s)$. 
Hence $w_n$ is obtained as a convolution of two subexponential sequences 
asymptotically equivalent to $C_1\pr(\eta_1>n)$ and $C_2\pr(\eta_1>n)$ and therefore 
behaves as $C_3\pr(\eta_1>n)$ for some $C_3$. 
This implies the statement of the lemma. 
\end{proof}    

In the following lemma we complete the proof of Theorem~\ref{thm:subex_for_R}.
\begin{lemma}\label{lem:v_n}
Assume that $F\in \mathcal S^*$, where $F(x)=\pr(\eta_1\le x)$. 
Then, for any $x>x_0$, 
\[
\pr_x(T^{(R)}_{x_0}>n)\sim u(x)\pr(\eta_1>n),\quad n\to \infty,      
\]
where the function $u(x)=\e_x[T^{(R)}_{x_0}]$ has been computed in 
\eqref{eq:mean.hitting.3}.
\end{lemma}   
\begin{proof}
First we derive a lower bound. For every $N\ge1$ one has 
\begin{align*}
\{T^{(R)}_{x_0}>n\}\supseteq
&\bigcup_{k=1}^N\{T^{(R)}_{x_0}>n,\eta_k>x_0+n\}\\
&=\bigcup_{k=1}^N\{T^{(R)}_{x_0}>k-1,\eta_k>x_0+n\}.
\end{align*}
Therefore, by the inclusion-exclusion argument,
\begin{align*}
\pr_x(T^{(R)}_{x_0}>n)
&\ge\sum_{k=1}^N\pr_x(T^{(R)}_{x_0}>k-1)\pr(\eta_k>x_0+n)\\
&\hspace{1cm}
-\sum_{k=1}^{N-1}\pr_x(T^{(R)}_{x_0}>k-1)\pr(\eta_k>x_0+n)
\sum_{j=k+1}^N\pr(\eta_j>x_0+n)\\
&\ge(1-N\pr(\eta_1>x_0+n))\pr(\eta_1>x_0+n)\sum_{k=1}^N\pr_x(T^{(R)}_{x_0}>k-1).
\end{align*}
This implies that  
$$
\liminf_{n\to\infty}\frac{\pr_x(T^{(R)}_{x_0}>n)}{\pr(\eta_1>x_0+n)}
\ge\sum_{k=1}^N\pr_x(T^{(R)}_{x_0}>k-1).
$$

Letting $N$ to infinity we obtain
\begin{equation}\label{eq:tail.lower}
\liminf_{n\to\infty}\frac{\pr_x(T^{(R)}_{x_0}>n)}{\pr(\eta_1>x_0+n)}
\ge\sum_{k=1}^\infty\pr_x(T^{(R)}_{x_0}>k-1)=\e_x[T^{(R)}_{x_0}]=u(x).
\end{equation}
 
 We next derive the corresponding asymptotic precise upper bound for $v_n$.  Fix $\varepsilon>0$. From Lemma~\ref{lem:tail.upper} and from the subexponentiality of $\pr(\eta_1>x_0+n)$ we conclude that there exists $N$ such that 
 \[
 \sum_{m=N+1}^{n-N}
 v_{n-m-1}\pr(\eta_1\in (x_0+m,x_0+{n-1}])
 \le 
 C 
 \sum_{m=N}^{n-N}
 \overline F(n-m)\overline F(m)
 \le \frac{\varepsilon}{2} \overline F(x_0+n)
 \]
 for all $n\ge 2N$.
 Also, since $F\in \mathcal S^*$ for any fixed $i$, 
 $$
 \pr(\eta_1\in (x_0+n-i,x_0+n])=o(\overline F(n))
 $$
 and therefore for all $n\ge 2N$,
 \[
 \sum_{m=n-N}^{n-2}
 v_{n-m-1}\pr(\eta_1\in (x_0+m,x_0+{n-1}])
 \le \frac{\varepsilon}{2} \overline F(x_0+n).
 \]
 Combining these estimates with the representation \eqref{eq:tail.8}, we get 
 $$
 v_n\le (1+\varepsilon)c_n+\sum_{m=0}^N d_mv_{n-m-1},
 \quad n\ge 2N,
 $$
 where the sequences $\{c_n\}$ and ${d_n}$ are defined in the proof of Lemma~\ref{lem:tail.upper}.
 
Set now $w_n^{(N)}=v_n$ for $n<2N$ and 
$$
 w_n^{(N)}= (1+\varepsilon)\pr(\eta_1>x_0+n-1)
 +\sum_{m=0}^N d_mw^{(N)}_{n-m-1}.
$$
Clearly $v_n\le w_n^{(N)}$ for all $n$. 

Set also $\widehat{w}^{(N)}(s)=\sum_{n=2N}^\infty w_n^{(N)}s^n$
and $\widehat{c}^{(N)}(s)=\sum_{n=2N}^\infty c_ns^n$.
Then one has 
\begin{align*}
\widehat{w}^{(N)}(s)
&=(1+\varepsilon)\widehat{c}^{(N)}(s)
+s\sum_{m=0}^N d_ms^m\sum_{n=2N}^\infty w^{(N)}_{n-m-1}s^{n-m-1}\\
&=(1+\varepsilon)\widehat{c}^{(N)}(s)
+s\sum_{m=0}^N d_ms^m\sum_{n=2N-m-1}^{2N-1}w^{(N)}_{n-m-1}s^{n-m-1}\\
&\hspace{2cm}+s\widehat{w}^{(N)}(s)\sum_{m=0}^N d_ms^m.
\end{align*}
Therefore,
$$
\widehat{w}^{(N)}(s)
=\frac{(1+\varepsilon)\widehat{c}^{(N)}(s)
+s\sum_{m=0}^N d_ms^m\sum_{n=2N-m-1}^{2N-1}w^{(N)}_{n-m-1}s^{n-m-1}}{1-s\sum_{m=0}^N d_ms^m}.
$$
This implies that 
$$
w_n^{(N)}\sim\left(\frac{1+\varepsilon}{1-\sum_{m=0}^Nd_m}\right)c_n.
$$
Consequently,
$$
\limsup_{n\to\infty}\frac{v_n}{c_n}\le \frac{1+\varepsilon}{1-\sum_{m=0}^Nd_m}\le \frac{1+\varepsilon}{1-\sum_{m=0}^\infty d_m}
=(1+\varepsilon)u(x_0+1).
$$
Letting $\varepsilon\to0$, we get 
$$
\limsup_{n\to\infty}\frac{v_n}{c_n}\le u(x_0+1).
$$

Combining this with the lower bound \eqref{eq:tail.lower}, we conclude that
\begin{equation}
\label{eq:vn-asymp}
v_n\sim u(x_0+1)\pr(\eta_1>n),\quad n\to\infty.
\end{equation}

Assume now that $x\in(x_0+k,x_0+k+1]$ for some $k\ge1$. Then, by Proposition~\ref{prop:tail.recursion},
\begin{align*}
\pr(T^{(R)}_{x_0}>n)
=v_n+\sum_{m=1}^kv_{n-m}\prod_{j=0}^{m-1}\pr(\eta_1\le x_0+j).
\end{align*}
Using now \eqref{eq:vn-asymp}, we conclude that 
$$
\pr(T^{(R)}_{x_0}>n)\sim u(x_0+1)\left(1+\sum_{m=1}^k\prod_{j=0}^{m-1}\pr(\eta_1\le x_0+j)\right)\pr(\eta>n).
$$
Noting that $u(x_0+1)\left(1+\sum_{m=1}^k\prod_{j=0}^{m-1}\pr(\eta_1\le x_0+j)\right)=u(x)$ we complete the proof.
\end{proof}
\section{Proof of Theorem~\ref{thm:subex_for_X}}
\label{sec:subex_for_X}
The lower bound for the tail of $T_{x_0}^{(X)}$ can be obtained by exactly the same arguments as the lower bound for $T_{x_0}^{(R)}$ in Lemma~\ref{lem:v_n}.

We turn to the corresponding upper bound. Set 
$c=\frac{1}{2\sum_{j=1}^\infty j^{-2}}$. For every $y\ge x_0$ we define the events
$$
\left\{\xi_k\le \frac{A^{n-k}}{(n-k+1)^2}cy\right\},\quad k\le n.
$$
On the intersection of these sets one has
\begin{align*}
X_n&=a^nX_0+\sum_{k=1}a^{n-k}\xi_k\\
&\le a^nX_0+\sum_{k=1}^n\frac{cy}{(n-k+1)^2}
\le a^nX_0+y/2.
\end{align*}
If $n$ is sufficiently large, say $n\ge n_0=n_0(X_0)$ then we infer that
$X_n\le y$. Therefore,
\begin{equation}
\label{eq:ub1}
\pr_x(X_n>y,T_{x_0}^{(X)}>n)
\le\sum_{k=1}^n\pr_x(T_{x_0}^{(X)}>k-1) c_{n-k}(y),
\quad n\ge n_0,
\end{equation}
where
$$
c_j(y):=\pr\left(\xi_1>\frac{A^{j}}{(j+1)^2}cy\right),\quad j\ge0.
$$

We first use this estimate with $y=x_0$. In this case we have
$$
\pr_x(T_{x_0}^{(X)}>n)
\le\sum_{k=1}^n\pr_x(T_{x_0}^{(X)}>k-1) c_{n-k}(x_0),
\quad n\ge n_0.
$$
Consider the sequence $\{w_n\}$ which is defined via the
recursion 
$$
w_n=\sum_{k=1}^nw_{k-1} c_{n-k}(x_0)
$$
with initial condition $w_0=w_1=\ldots=w_{n_0-1}=1$. Then clearly 
\begin{equation}
\label{eq:ub2}
\pr_x(T_{x_0}^{(X)}>n)\le w_n,\quad n\ge0.
\end{equation}
It is immediate from the definition of $\{w_n\}$ that 
\begin{align*}
\sum_{n=n_0}^\infty w_ns^n
&=\sum_{n=n_0}^\infty s^n\sum_{k=1}^nw_{k-1} c_{n-k}(x_0)\\
&=\sum_{n=n_0}^\infty s^n\sum_{k=1}^{n_0-1}w_{k-1} c_{n-k}(x_0)
+\sum_{n=n_0}^\infty s^n\sum_{k=n_0}^nw_{k-1} c_{n-k}(x_0)
\end{align*}
Setting
$$
d_n(x_0):=\sum_{k=1}^{n_0-1}w_{k-1} c_{n-k}(x_0)
$$
and interchanging the order of summation in the second series, we
conclude that 
$$
\sum_{n=n_0}^\infty w_ns^n
=\frac{\sum_{n=n_0}^\infty d_n(x_0)s^n}{1-s\sum_{j=0}^\infty c_j(x_0)s^j}.
$$
Using once again the results from \cite{AFK03}, we infer that 
$$
w_n\sim C\pr(\eta_1>n)
$$
provided that $\sum_{j=0}^\infty c_j(x_0)<1$. Combining this with \eqref{eq:ub2}, we obtain 
\begin{equation}
\label{eq:ub3}
\pr_x(T_{x_0}^{(X)}>n)\le C\pr(\eta_1>n),\quad n\ge0.
\end{equation}
Using Lemma~\ref{lem:down}, we conclude that \eqref{eq:ub3} is valid for all $x_0$ such that $\pr(ax_0+\xi_1<x_0)$ is strictly positive.

Combining now \eqref{eq:ub1},\eqref{eq:ub3} and recalling that the sequences $\pr(\eta_1>n)$ and $c_n(y)$ are subexponential, we conclude that 
\begin{equation}
\label{eq:ub4}
\limsup_{n\to\infty}\frac{\pr_x(X_n>y,T_{x_0}^{(X)}>n)}{\pr(\eta_1>n)}
\le \e_x[T_{x_0}^{(X)}]+C(y),
\end{equation}
where
$$
C(y):=\sum_{k=0}^\infty c_k(y).
$$
This quantity is finite due to the assumption $\e\eta_1<\infty$. Furthermore,
$C(y)\to0$ as $y\to\infty$.

Fix now a integer-valued sequence $N_n\to\infty$ such that
$\pr(\eta_1>n)\sim\pr(\eta_1>n-N_n)$. By the monotonicity of the chain
$\{X_n\}$,
\begin{align*}
&\pr_x(T_{x_0}^{(X)}>n)\\
&=\pr_x(X_{n-N_n}>y,T_{x_0}^{(X)}>n)+\pr_x(X_{n-N_n}\le y,T_{x_0}^{(X)}>n)\\
&\le \pr_x(X_{n-N_n}>y,T_{x_0}^{(X)}>n-N_n)
+\pr_x(T_{x_0}^{(X)}>n-N_n)\pr_y(T_{x_0}^{(X)}>N_n).
\end{align*}
Applying \eqref{eq:ub3} and \eqref{eq:ub4}, we get 
\begin{align*}
\limsup_{n\to\infty}\frac{\pr_x(T_{x_0}^{(X)}>n)}{\pr(\eta_1>n)}
&\le \e_x[T_{x_0}^{(X)}]+C(y)+C\lim_{n\to\infty}\pr_y(T_{x_0}^{(X)}>N_n)\\
&=\e_x[T_{x_0}^{(X)}]+C(y).
\end{align*}
Letting now $y\to\infty$ and recalling that $\lim_{y\to\infty}C(y)=0$, we finally obtain
$$
\limsup_{n\to\infty}\frac{\pr_x(T_{x_0}^{(X)}>n)}{\pr(\eta_1>n)}
\le \e_x[T_{x_0}^{(X)}].
$$
Thus, the proof is complete.

\end{document}